\documentclass[12pt,a4paper]{amsart}
\usepackage{amsfonts}
\usepackage{amssymb}
\usepackage{amsmath}
\usepackage{amsthm} 
\usepackage{cite}
\usepackage{enumitem}
\usepackage{tikz}
\usepackage{subfigure}
\usepackage{latexsym}
\usepackage{dsfont}

\theoremstyle{plain}
\newtheorem{thm}{Theorem}

\newtheorem{lem}[thm]{Lemma}
\newtheorem{prop}[thm]{Proposition}
\newtheorem{cor}[thm]{Corollary}
\newtheorem{rem}[thm]{Remark}

\renewcommand{\b}[1]{\textcolor{blue}{#1}}

\newcommand{\sph}[1]{\mathbb{S}^{#1}}
\providecommand{\ind}{\mathds{1}} 
\providecommand{\les}{\lesssim}

\renewcommand{\S}{\mathbb{S}^1}
\providecommand{\N}{\mathbb{N}}
\providecommand{\R}{\mathbb{R}}

\providecommand{\C}{\mathbb{C}}
\providecommand{\eps}{\varepsilon}
\providecommand{\ov}{\overline}

\DeclareMathOperator{\supp}{supp}

\renewcommand{\qed}{\hfill $\Box$}

\setitemize{itemsep=+2pt} 
\setenumerate{itemsep=+2pt}

\begin{document}


\allowdisplaybreaks

\title{Optimal weighted Fourier restriction estimates for the sphere in 2D}

\author{Rainer Mandel}
\email{Rainer.Mandel@gmx.de}


\begin{abstract}
  We prove weighted versions of the 2D Restriction Conjecture for the unit sphere in
  $\R^2$. Our results involve the weight functions $(1+|x|)^\alpha(1+|y|)^\beta$ and  $(1+|x|+|y|)^\gamma$
  with $\alpha,\beta,\gamma\geq 0$.
\end{abstract}

\maketitle
\allowdisplaybreaks
   
\section{Introduction}

  In this paper we investigate  weighted $L^p$-$L^q$-estimates for the Fourier
  Restriction operator  
  $$
    (\mathcal R g)(\xi) 
    := \hat g(\xi) 
    := \frac{1}{2\pi} \int_{\R^2} g(x,y) e^{-i(x,y)\cdot\xi} \,d(x,y)
    \quad\text{for }\xi\in\S
  $$
  of the two-dimensional unit sphere $\S := \{\xi\in\R^2: |\xi|=1\}$.
  Its adjoint $\mathcal R^*$ is   the Fourier Extension operator  given by  the formula
  $$
    (\mathcal R^*F)(x,y) = \frac{1}{2\pi} \int_{\S} e^{i(x,y)\cdot\omega}
    F(\omega)\,d\sigma(\omega) \qquad\text{for } (x,y)\in\R^2.
  $$  
  Here, $\sigma$ denotes the canonical surface measure on $\S$.
  The mapping properties of $\mathcal R$ induce the mapping properties
  of $\mathcal R^*$ and vice versa, so no generality is lost by investigating $\mathcal R^*$.  
  The following fact, which is the celebrated Restriction Conjecture in two spatial dimensions, was proved by
  Fefferman \cite{Fefferman} and Zygmund \cite{Zygmund} about 50 years ago.
  
\begin{thm}[Fefferman, Zygmund] \label{thm:FeffermanZygmund}
   $\mathcal R:L^p(\R^2)\to L^\mu(\sph{1})$ is bounded if and only if $1\leq p<\frac{4}{3}$ and
  $3\mu\leq p'$.  Equivalently, 
  $\mathcal R^*:L^r(\sph{1})\to L^q(\R^2)$ is bounded if and only if $q\geq 3r'$ and $q>4$. 
\end{thm}

 Here, the whole difficulty is to prove the endpoint estimates $q=3r'$ 
 in view of the trivial embeddings of Lebesgue spaces on the unit sphere. 
 Our aim is to generalize this result to a weighted setting that contains another result from the
 literature. Bloom and Sampson~\cite[Theorem 2.11]{BloomSampson} proved optimal $L^2$-bounds for the operator 
  $$
    (\mathcal R_{\alpha,\beta}^*F)(x,y) 
    := (\mathcal R^* F)(x,y) (1+|x|)^{-\alpha}(1+|y|)^{-\beta},\qquad x,y\in\R,  
  $$
  in the special case $\alpha=\beta$.  Their result reads as follows. 
 
 \begin{thm}[Bloom, Sampson] \label{thm:BloomSampson} 
   $\mathcal R_{\alpha,\alpha}^*:L^2(\sph{1})\to L^2(\R^2)$
   is bounded if and only if $\alpha\geq \frac{1}{3}$. 
 \end{thm}
 
 We merge these two results in an optimal weighted version of the Restriction
 Conjecture.    
  
\begin{thm} \label{thm:main}
  Assume $0\leq \alpha,\beta<\infty$ and $1\leq r\leq \infty, 0<q<\infty$. Then
  $\mathcal R_{\alpha,\beta}^*:L^r(\sph{1})\to L^q(\R^2)$ is bounded provided that
  $\alpha+\beta>\frac{2}{q}-\frac{1}{2}$ as well as
  \begin{itemize}
      \item[(i)] $\max\{\alpha,\beta\}\geq \frac{1}{q}$ and  $2\min\{\alpha,\beta\}>\frac{2}{q}-\frac{1}{r'}$
      \qquad or
      \item[(ii)] $\max\{\alpha,\beta\}<\frac{1}{q}$ and  
       $\alpha+\beta+\min\{\alpha,\beta\}>  \frac{3}{q}-\frac{1}{r'}$
    \end{itemize}      
  holds. If additionally $1<r\leq q$ holds, then the operator is also bounded for
  \begin{itemize}
      \item[(iii)] $\max\{\alpha,\beta\}>\frac{1}{q}$ and  $2\min\{\alpha,\beta\}=\frac{2}{q}-\frac{1}{r'}$
      \qquad or
      \item[(iv)] $\max\{\alpha,\beta\}<\frac{1}{q}$ and 
       $\alpha+\beta+\min\{\alpha,\beta\}=\frac{3}{q}-\frac{1}{r'}$.
    \end{itemize}     
  In the case $q=\infty$ the operator is bounded for all $r\in [1,\infty],\alpha,\beta\in [0,\infty)$. \\  
  These conditions are optimal.   
\end{thm} 

 Note that the special case $\alpha=\beta=0$ reproduces Theorem~\ref{thm:FeffermanZygmund} whereas the ansatz
 $\alpha=\beta$ and $q=r=2$ leads to Theorem~\ref{thm:BloomSampson}. The estimates (i),(ii) will be called
 nonendpoint estimates whereas (iii),(iv) are referred to as endpoint estimates.

 \begin{figure}[htbp]
    \centering
    \includegraphics[scale=0.7]{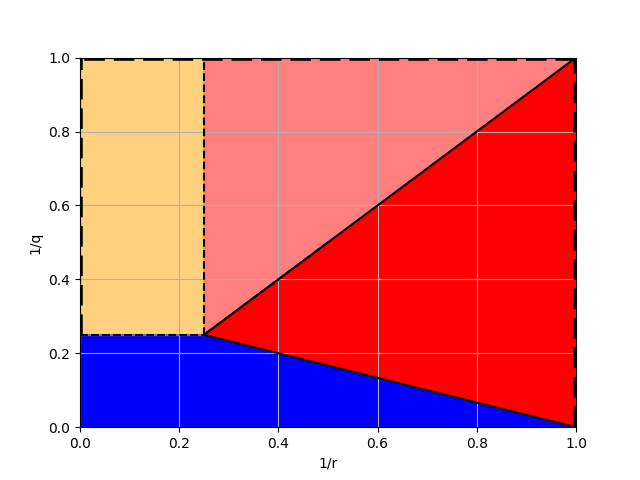}
    \caption{Riesz diagram for Theorem~\ref{thm:main} in the special case $\alpha=\beta$: the indicated
    $(\frac{1}{r},\frac{1}{q})$-regions indicate different optimal conditions on $\alpha$, blue: $\alpha\geq
    0$, orange: $\alpha>\frac{1}{q}-\frac{1}{4}$, red: $\alpha\geq \frac{1}{q}-\frac{1}{3r'}$, 
    pale red: $\alpha> \frac{1}{q}-\frac{1}{3r'}$, 
    dashed lines: strict inequalities, solid lines: non-strict inequalities.}
    \label{fig:Theorem3}
  \end{figure}

  \begin{rem} \label{rem:general} ~
  \begin{itemize}
    \item[(a)] The conditions on $q,r$ reflect the simple observation that more restrictive assumptions on the
    input data (larger $r$) and stronger weights (larger $\alpha,\beta$) yield a wider range of admissible
    exponents $q$.
    Moreover, since the estimate for $q=\infty$ holds for all $\alpha,\beta \in [0,\infty)$ and $r\in
    [1,\infty]$, an interpolation argument shows that the estimates for the lower exponents $q$ imply the ones
    for larger $q$. Formally, if $\mathcal R^*_{\alpha,\beta}:L^r(\sph{1})\to L^q(\R^2)$ is bounded, then it
    is automatically bounded for parameters $(\alpha,\beta,r,q)$ replaced by 
    $(\alpha+\eps_\alpha,\beta+\eps_\beta,r+\eps_r,q+\eps_q)$ whenever
    $\eps_\alpha,\eps_\beta,\eps_r,\eps_q\geq 0$.     
    \item[(b)] The sufficiency part of Theorem~\ref{thm:main} carries over to weight
    functions of the form $w_1(x)w_2(y)$ where 
    $w_1\in X^\alpha, w_2\in X^\beta$ and $X^\gamma := L^{1/\gamma,\infty}(\R)\cap L^\infty(\R)$.
    Even more generally, this is true for weights $w(x,y)$ belonging to the
    intersection of mixed norm spaces $X_x^\alpha(X_y^\beta) \cap X_y^\beta(X^\alpha_x)$ where $x,y$
    indicate one-dimensional integration variables. To see this it suffices to adapt
    the proofs of Proposition~\ref{prop:Pitttype} and 
    Lemma~\ref{lem:weaktypebounds_largealpha}. The subsequent interpolation procedure based on these two
    auxiliary results is the same.
    \item[(c)] The a priori assumption $\alpha,\beta\geq 0$ has been added for the sake of clarity. In fact,
    it is a necessary condition for the estimates to hold. This follows from the equivariance
    property $\mathcal R^*(Fe^{ih\cdot})(x,y)=(\mathcal R^*F)(x+h_1,y+h_2)$ where $h=(h_1,h_2)\in\R^2$, see
    \cite[p.88]{BloomSampson}.
  \end{itemize}
  \end{rem}
  
  We extend our analysis to the  weighted Fourier Extension operator 
  $$
    \mathfrak R_\gamma^*:L^r(\S)\to L^q(\R^2),\quad (\mathfrak R_\gamma^* F)(x,y):=  (\mathcal R^* F)
    (1+|x|+|y|)^{-\gamma}.
  $$ 
  In \cite[Theorem~4.8]{BloomSampson} Bloom and Sampson provided an almost complete study of its mapping
  properties.
  
  \begin{thm}[Bloom, Sampson] \label{thm:BloomSampson2}
    Assume $2\leq q<\infty,1\leq r\leq \infty$. Then the following assumptions are sufficient for 
    $\mathfrak R_\gamma^*:L^r(\S)\to L^q(\R^2)$ to be bounded:    
  \begin{itemize}
    \item[(a)] if $\frac{3}{q}-\frac{1}{r'}\leq 0,4<q<\infty$, then $\gamma\geq 0$,
    \item[(b)] if  $2\leq q\leq 4,q\leq r$, then $\gamma>\frac{2}{q}-\frac{1}{2}$, 
    \item[(c)] if $q=2>r$, then $\gamma\geq \frac{1}{r}$,
    \item[(d)] if $q=r'>2$, then $\gamma>\frac{1}{q}$,
    \item[(e)] if $2<q<r'$, then $\gamma\geq \frac{2}{q}-\frac{1}{r'}$,
    \item[(f)] else $\gamma>\frac{3}{2q}-\frac{1}{2r'}$.
  \end{itemize}
    Given the assumptions on $q,r$, the range for $\gamma$ is optimal in (a),(b),(c),(e) and optimal possibly
    up to the endpoint cases $\gamma=\frac{1}{q}$ in (d) and $\gamma=\frac{3}{2q}-\frac{1}{2r'}$ in (f).
  \end{thm}
  
  Our intention is to include exponents $0<q<2$ and close the gap about the endpoint cases
  in (d),(f). It turns out that this can be done along the lines of the proof of Theorem~\ref{thm:main}.
   
  \begin{thm} \label{thm:main2}
    Assume $0\leq \gamma<\infty$ and $1\leq r\leq \infty, 0<q<\infty$. Then $\mathfrak
    R_{\gamma}^*:L^r(\sph{1})\to L^q(\R^2)$ is bounded provided that $\gamma>\frac{2}{q}-\frac{1}{2}$ as well
    as
   \begin{itemize}
      \item[(i)] $\gamma>\max\{\frac{3}{2q}-\frac{1}{2r'},\frac{2}{q}-\frac{1}{r'}\}$.
   \end{itemize}     
   If additionally $1<r\leq q$ holds, then the operator is also bounded for
   \begin{itemize}
      \item[(ii)] $\gamma=\max\{\frac{3}{2q}-\frac{1}{2r'},\frac{2}{q}-\frac{1}{r'}\}$ provided that 
      $q\neq r'$.
   \end{itemize}       
   In the case $q=\infty$ the operator is bounded for all $r\in [1,\infty],\gamma\in [0,\infty)$. \\  
   These conditions are optimal.   
 \end{thm}
 
 Case dinstinctions show that Theorem~\ref{thm:main2}  implies Theorem~\ref{thm:BloomSampson2} 
 and fill the aforementioned gap: the equality case in (d) does not belong to the boundedness range whereas
 the case $\gamma=\frac{3}{2q}-\frac{1}{2r'}$ with $q\neq r'$ in (f) does. This is visualized in
 Figure~\ref{fig:Theorem6} below.

 \begin{figure}[htbp]
    \centering
    \includegraphics[scale=0.7]{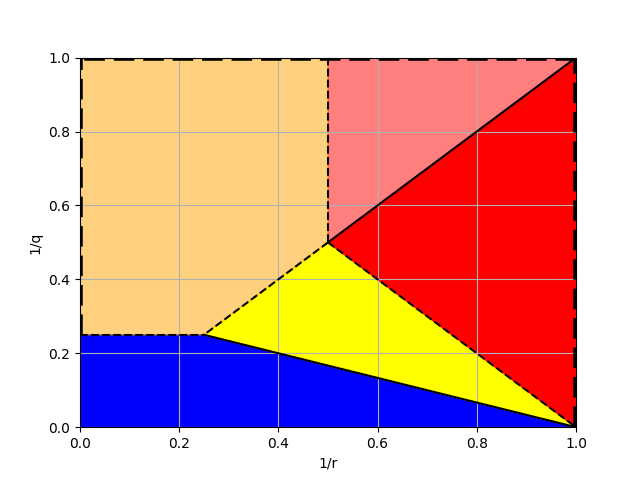}
    \caption{Riesz diagram for Theorem~\ref{thm:main2}:
    blue: $\gamma\geq 0$, orange: $\gamma>\frac{2}{q}-\frac{1}{2}$,
    yellow: $\gamma\geq \frac{3}{2q}-\frac{1}{2r'}$,
    red: $\gamma\geq \frac{2}{q}-\frac{1}{r'}$, 
    pale red: $\gamma> \frac{2}{q}-\frac{1}{r'}$, 
    dashed lines: strict inequalities, solid lines: non-strict inequalities. Our refinements of
    Theorem~\ref{thm:BloomSampson2} concern the upper half, the endpoint case 
    $\gamma=\frac{3}{2q}-\frac{1}{2r'}$ in the yellow region and the endpoint case
    $\gamma=\frac{1}{q}$ on the diagonal bottom right.}
    \label{fig:Theorem6}
  \end{figure}

\begin{rem}
  Our counterexamples show that the conditions
  $\gamma>\frac{2}{q}-\frac{1}{2}$ respectively $\alpha+\beta>\frac{2}{q}-\frac{1}{2}$ originate from the
  constant density on the sphere whereas all other conditions on $\alpha,\beta,\gamma$ are sharp in view of Knapp-type
  counterexamples. For the same reason as in Remark~\ref{rem:general}(c)  the a priori assumption $\gamma\geq
  0$ is actually a necessary condition for the estimates to hold.
\end{rem}

%

\section{Preliminaries}

  In this section we collect some known facts about real and complex interpolation theory
  for Lorentz spaces and mixed Lorentz spaces.  
  To put this into the abstract framework 
  used in the paper \cite{GrafakosMastylo} by
  Grafakos-Masty{\l}o, we start with quasi-Banach (function) lattices $X$ on a given measure space.  
  Such a lattice is a vector space of measurable functions that are finite almost everywhere. 
  By definition, it is complete with respect to a quasi-norm $\|\cdot\|_X$ and  solid, i.e.,  
  $|f|\leq |g|$ almost everywhere and $g\in X$ implies $f \in X$ with $\|f\|_{X} \leq\|g\|_{X}$.
  A quasi-Banach lattice $X$ is said to be maximal   whenever $0 \leq f_{n}
  \uparrow f$ almost everywhere with $f_{n} \in X$ and $\sup _{n \geq 1}\left\|f_{n}\right\|_{X}<\infty$
  implies that $f \in X$ and $\left\|f_{n}\right\|_{X} \rightarrow\|f\|_{X}$. In other words, we have some
  sort of Monotone Convergence Theorem in $X$. It is called $p$-convex for $0<p<\infty$ if there exists a
  constant $C>0$ such that for any $n\in\N$ and $f_{1}, \ldots, f_{n} \in X$ we have 
  $$
  \left\|\left(\sum_{k=1}^{n}\left|f_{k}\right|^{p}\right)^{\frac{1}{p}}\right\|_{X} \leq
  C\left(\sum_{k=1}^{n}\left\|f_{k}\right\|_{X}^{p}\right)^{\frac{1}{p}} .
$$
A quasi-Banach lattice is said to have nontrivial convexity whenever it is $p$-convex for some
$0<p<\infty$. For us the important fact is that Lorentz spaces and  mixed Lorentz spaces 
belong to this class of spaces. To state this more precisely, we say that a tuple of exponents $(p,r)$ is 
Lorentz-admissible if $0<p<\infty,0<r\leq \infty$ or $p=r=\infty$. 
We then write $L^{p,r}:=L^{p,r}(\R^d)$ for the standard Lorentz spaces and $L^{\vec p,\vec
r}:=L^{p,r}(\R^{d-k};L^{p,r}(\R^k))$ for mixed Lorentz spaces where $d\in\N$ and
$k\in\{1,\ldots,d-1\}$.
Recall that the latter is a quasi-Banach lattice equipped with the norm 
$$
  f\mapsto \Big\| \|f(x,y)\|_{L_y^{p,r}(\R^k)} \Big\|_{L_x^{p,r}(\R^{d-k})}
$$ 
where the subscripts  indicate the integration variable. 

  \begin{prop} 
    Let  $d\in\N$ and $k\in\{1,\ldots,d-1\}$ and let $(p,r)$ be Lorentz-admissible. Then the Lorentz space
    $L^{p,r}$ and the  mixed Lorentz space $L^{\vec p,\vec r}$ are maximal quasi-Banach lattices with
    nontrivial convexity.
  \end{prop} 

Note that in the case $r<\infty$ the maximality of Lorentz spaces follows from Fatou's Lemma and in the case
of mixed Lorentz spaces one uses that a mixed quasi-Banach lattice $X(Y)$ is maximal if both
$X,Y$ are maximal. We omit the details of the proof and turn towards its relevance for interpolation
theoretical applications. Real interpolation theory for
abstract quasi-Banach spaces is explained in \cite[Section 3.11]{BL}
and the important special case of Lorentz spaces can be found in the sections 5.2 and 5.3 in this book. 
An analogous theory for mixed Lorentz spaces, however, appears to be missing in the literature. 
For that reason we prove one basic result in Proposition~\ref{prop:embedding} below.
Moreover, we will use (complex) interpolation for analytic families of linear operators, which we 
 call  ``Stein interpolation'' in view of the fundamental contribution by
Stein~\cite{Stein1956} on this matter. Instead of the original version dealing with Banach spaces we use the
nontrivial extension to quasi-Banach lattices due to Grafakos and Masty{\l}o~\cite{GrafakosMastylo}.  

\medskip

 For a given couple of Banach spaces $(X_0,X_1)$ \cite[Section 2.3]{BL}  the symbol $[X_0,X_1]_\theta$
   denotes the standard complex interpolation space as defined in the book of Bergh-L\"ofstr\"om 
 \cite[p.88]{BL}. It is a Banach space equipped with the norm 
 $$
   \|f\|_{[X_0,X_1]_\theta} 
   = \inf_{F(\theta)=f,\, F\in\mathcal F}  
    \max\{\|F\|_{L^\infty(i\R;X_0)},\|F\|_{L^\infty(1+i\R;X_1)}\}. 
 $$ 
 Here, $\mathcal F$ denotes the set of functions $F:\ov S\to X_0+X_1$ that are bounded and
 continuous on $\ov S$, with $F:i\R\to X_0,F:1+i\R\to X_1$ vanishing at $\pm i\infty,1\pm i\infty$  and being
 analytic in the strip $S:=\{s\in\C: 0<\Re(s)<1\}$.
 For our application, all we need to know is the embedding $[X_0,X_1]_\theta\supset (X_0,X_1)_{\theta,1}$ from
 \cite[Theorem~4.7.1]{BL}.
 
 \medskip
 
 We shall moreover need another sort of interpolation space.
If $Y_{0}$ and $Y_{1}$ are quasi-Banach lattices over the same measure space and $0<\theta<1$, then their
Calder\'{o}n product $Y_{0}^{1-\theta} Y_{1}^{\theta}$  is defined as the vector space of all
measurable functions $g$ such that $|g| \leq\left|g_{0}\right|^{1-\theta}\left|g_{1}\right|^{\theta}$
holds almost everywhere  for some $g_0\in Y_0,g_1\in Y_1$. It 
is again a quasi-Banach lattice equipped with the quasi-norm 
$$ 
    \|g\|_{Y_{0}^{1-\theta} Y_{1}^{\theta}}=\inf
   \left\{\left\|g_{0}\right\|_{Y_{0}}^{1-\theta}\left\|g_{1}\right\|_{Y_{1}}^{\theta}: |g|
   \leq\left|g_{0}\right|^{1-\theta}\left|g_{1}\right|^{\theta}  \text {almost everywhere}, g_0\in
   Y_0, g_1\in Y_1\right\} .
$$

\medskip

Following~\cite{GrafakosMastylo} we finally introduce the notion of an admissible analytic family of linear
operators between couples $(X_0,X_1)$ and $(Y_0,Y_1)$ as introduced above. Let $\mathcal{X}$ be a subspace of
$X_0\cap X_1$. We assume that for every $z \in \bar{S}$ there is a linear map $T_z:\mathcal X\to Y_0\cap Y_1$ 
such that $T_{z}f$ is a complex-valued measurable  function that is finite almost everywhere for all
$f\in\mathcal X$. The family $(T_{z})_{z \in \bar{S}}$ is then said to be analytic if for any
$f\in\mathcal X$ and for almost all $x$ the function $z \mapsto T_{z}f(x)$ is  
analytic in $S$ and continuous on $\bar{S}$. If additionally $(z,x) \mapsto T_{z}f(x)$ is 
measurable for every $f \in \mathcal{X}$, then $(T_z)_{z\in\bar S}$ is called an admissible analytic family.
We shall need the following simplified version of Theorem 1.1 in \cite{GrafakosMastylo}. 

  \begin{thm}[Grafakos-Masty{\l}o] \label{thm:GrafMas}
    Let $(X_{0}, X_{1})$ be a couple of Banach spaces and let $(Y_{0},Y_{1})$ be a couple of
    maximal quasi-Banach lattices on a measure space such that $Y_0,Y_1$ have nontrivial convexity.
    Assume that $\mathcal X$ is a dense linear subspace of $X_{0} \cap X_{1}$ and that
    $(T_{z})_{z \in \bar{S}}$ is an admissible analytic family of linear operators $T_{z}:
    \mathcal X \rightarrow Y_{0}\cap Y_{1}$. Suppose that for every $f \in
    \mathcal X, t \in \mathbb{R}$ and $j=0,1$, 
    $$
	  \left\|T_{j+i t} f \right\|_{Y_{j}} \leq C_j \left\|f\right\|_{X_{j}}.  
	$$
    Then there is $C>0$ such that for all $f\in\mathcal X$, $s\in \mathbb{R}$ and $0<\theta<1$ we have 
    $$
	  \left\|T_{\theta+i s}f\right\|_{Y_0^{1-\theta} Y_1^{\theta}}
      \leq C    \|f\|_{[X_{0}, X_{1}]_\theta}. 
    $$ 
  \end{thm}
  
  \begin{rem} ~
    \begin{itemize}
      \item[(a)] Interpolation theorems in special quasi-Banach spaces can be found
      in~\cite{Sagher69,CwikelSagher1988}. However, none of those applies to mixed Lorentz spaces $L^{\vec
      p,\vec r}$ with $0<p<1$. This is why we need the rather recent result from \cite{GrafakosMastylo}.
      \item[(b)] In the context of Lorentz spaces $Y_0,Y_1$ the Calder\'{o}n product
      $Y_0^{1-\theta}Y_1^\theta$ is a Lorentz space for most Lorentz-admissible exponents, see
      \cite[Corollary~4.1]{GrafakosMastylo}. Moreover, under quite general assumptions
      on Banach spaces $Y_0,Y_1$ we have $Y_0^{1-\theta}Y_1^\theta = [Y_0,Y_1]_\theta$ 
      and suitable analogues are known for quasi-Banach lattices, see \cite[pp. 2-3]{Yuan2016}. In that
      sense Theorem~\ref{thm:GrafMas} generalizes the original result by Stein~\cite{Stein1956} (under
      slightly stronger assumptions on $T_z$).
    \end{itemize}
  \end{rem}

%

  In our application, it will turn out convenient to embed $[X_{0}, X_{1}]_\theta$ and $Y_0^{1-\theta}
  Y_1^{\theta}$ into real interpolation spaces. The embedding 
  $$
    (X_0,X_1)_{\theta,1}\subset [X_0,X_1]_\theta\subset (X_0,X_1)_{\theta,\infty}
  $$ 
  is well-known and turns out to be helpful. We now show the analogous statement for
  $Y_0^{1-\theta}Y_1^\theta$.
  
  \begin{prop} \label{prop:CaldprodEmbedding}
    Let $(Y_{0},Y_{1})$ be a couple of maximal quasi-Banach lattices on a measure space. Then
    $(Y_0,Y_1)_{\theta,1}\subset Y_0^{1-\theta}Y_1^\theta\subset (Y_0,Y_1)_{\theta,\infty}$  for
    $0<\theta<1$.
  \end{prop}
  \begin{proof}
   Set $Y:= Y_0^{1-\theta}Y_1^\theta$.  For all $g\in Y_0\cap Y_1$ we have by definition of the norm
   $$
     \|g\|_{Y}
     = \inf_{|g|\leq |g_0|^{1-\theta}|g_1|^\theta \atop 
     g_0\in Y_0, g_1\in Y_1} \|g_0\|_{Y_0}^{1-\theta}\|g_1\|_{Y_1}^\theta
     \leq \|g\|_{Y_0}^{1-\theta}\|g\|_{Y_1}^\theta
   $$ 
   by choosing $g_0:=g_1:=g$. So Theorem 3.11.4~(b) in~\cite{BL} gives 
   $(Y_0,Y_1)_{\theta,1}\subset Y$. \\  
   To prove the other embedding let $g\in Y$ and fix $t>0$. Choose $g_0\in Y_0,g_1\in
   Y_1$ such that $|g|\leq |g_0|^{1-\theta}|g_1|^\theta$ almost everywhere and
   $\|g_0\|_{Y_0}^{1-\theta}\|g_1\|_{Y_1}^\theta \leq 2\|g\|_Y$. Then
   we equally have
   $|g|\leq |\tilde g_0|^{1-\theta}|\tilde g_1|^\theta$ with $\tilde g_0=s^{-\theta} g_0 \in Y_0,\tilde
   g_1=s^{1-\theta}g_1 \in Y_1$ for any given $s>0$. We may choose $s:= \|g_0\|_{Y_0} (t\|g_1\|_{Y_1})^{-1}$
   and obtain
   \begin{align*}
     \max\left\{ t^{-\theta} \|\tilde g_0\|_{Y_0}, t^{1-\theta}\|\tilde g_1\|_{Y_1}\right\}
     &= \max\left\{ (st)^{-\theta} \| g_0\|_{Y_0}, (st)^{1-\theta}\|g_1\|_{Y_1}\right\} \\
     &= \|g_0\|_{Y_0}^{1-\theta}  \|g_1\|_{Y_1}^\theta \\
     &\leq 2     \|g\|_{Y}.
   \end{align*}
    So we conclude that for all $t>0$ there are $\tilde g_0\in Y_0,\tilde g_1\in Y_1$ satisfying the above
    inequality. Then Theorem 3.11.4~(a) in \cite{BL} gives $Y\subset (Y_0,Y_1)_{\theta,\infty}$,
    which is all we had to show.
  \end{proof}
   
   In the case of Lebesgue or Lorentz spaces $X_0,X_1,Y_0,Y_1$   
   standard real interpolation theory \cite[Theorem~5.3.1]{BL} allows to compute the associated real
   interpolation spaces. In the case of mixed Lorentz spaces, however, we need some partial extension of a result from
   \cite{Mandel2023Real}. 
  
  \begin{prop} \label{prop:embedding}
    Assume $d\in\N,k\in\{1,\ldots,d-1\}$,     
    $0<r\leq \infty$ and $0<q_0\neq q_1\leq \infty$. Then, for $0<\theta<1$ and
    $\frac{1}{q_\theta}=\frac{1-\theta}{q_0}+\frac{\theta}{q_1}$,
    \begin{align*}
      \text{(i)}\quad
      L^{q_\theta,r}(\R^d) 
      &=
      \left(L^{q_0,\infty}(\R^{d-k};L^{q_0,\infty}(\R^k)),L^{q_1,\infty}(\R^{d-k};L^{q_1,\infty}(\R^k))\right)_{\theta,r}, \\
      \text{(ii)}\quad L^{q_\theta,r}(\R^d) 
      &=
      \left(L^{q_0}(\R^d),L^{q_1,\infty}(\R^{d-k};L^{q_1,\infty}(\R^k))\right)_{\theta,r}, \\
      \text{(iii)}\quad L^{q_\theta,\infty}(\R^d) 
      &\supset
       L^{q_0}(\R^d)^{1-\theta}\left(L^{q_1,\infty}(\R^{d-k};L^{q_1,\infty}(\R^k))\right)^{\theta}.
    \end{align*}  
  \end{prop}
  \begin{proof} 
     Theorem~2.22 in \cite{ChenSun2020} provides the inequality  
     $\|g\|_{L^{q_\theta}}\les  \|g\|_{L^{\vec q_0,\vec\infty}}^{1-\theta}
     \|g\|_{L^{\vec q_1,\vec\infty}}^\theta$ for all $q_0,q_1\in (0,\infty]$ with $q_0\neq q_1$.
     As in the previous  Proposition this can be combined with \cite[Theorem 3.11.4]{BL}  to deduce,  for
     all $\mu\in [q_0,\infty]$, 
     $$
      L^{q_\theta} 
      \supset (L^{\vec q_0,\vec\infty}, L^{\vec q_1,\vec\infty})_{\theta,1}
      \supset (L^{\vec q_0,\vec\mu}, L^{\vec q_1,\vec\infty})_{\theta,1}.
    $$
     Exploiting this identity for $\theta\in\{\theta_0,\theta_1\}$ with $0<\theta_0<\theta_1<1$ 
     as well as
     the Reiteration Theorem \cite[Theorem 3.11.5]{BL}  we deduce the embedding
    \begin{align*}
      &L^{q_\theta,r}
      = (L^{q_{\theta_0}},L^{q_{\theta_1}})_{\vartheta,r}  
      \supset \left((L^{\vec q_0,\vec\mu}, L^{\vec q_1,\vec\infty})_{\theta_0,1} 
      , (L^{\vec q_0,\vec\mu}, L^{\vec q_1,\vec\infty})_{\theta_1,1}\right)_{\vartheta,r}
      \supset (L^{\vec
      q_0,\vec{\mu}}, L^{\vec q_1,\vec\infty})_{\theta,r} \\
      &\text{whenever}\quad
      \theta=(1-\vartheta)\theta_0+\vartheta\theta_1,\; 0<\vartheta,\theta_0\neq\theta_1<1,\; r\in (0,\infty].
    \end{align*}
    On the other hand, the embeddings
    $L^{\vec q_0,\vec{\mu}}\supset L^{q_0}$ and $L^{\vec q_1,\vec\infty}\supset L^{q_1}$ imply  
    $$
      (L^{\vec q_0,\vec{\mu}}, L^{\vec q_1,\vec\infty})_{\theta,r}
      \supset (L^{q_0}, L^{q_1})_{\theta,r}
      = L^{q_\theta,r}.
    $$
    Combining both inclusions we infer (i) by choosing $\mu=\infty$ and (ii) by choosing
    $\mu=q_0$.
    Finally, (iii) follows from (ii) and the embedding $(Y_0,Y_1)_{\theta,\infty}\supset
    Y_0^{1-\theta}Y_1^\theta$ that we  established in Proposition~\ref{prop:CaldprodEmbedding}.
  \end{proof}

\section{Proof of Theorem~\ref{thm:main} -- Sufficient conditions}
  
  In this section we analyze the mapping properties of $\mathcal R^*_{\alpha,\beta}:L^r(\S)\to L^q(\R^2)$. 
  Since the claim is trivial for $q=\infty$, we shall assume $0<q<\infty$. Moreover, without loss
  of generality, we assume $\alpha\geq\beta\geq 0$. We shall need the following auxiliary result.
    
\begin{prop} \label{prop:Pitttype}
  Assume $1\leq p\leq 2,0<q< \infty$ and 
  $0\leq \frac{1}{q}-\frac{1}{p'}\leq \beta<\infty$.
  Then
  $$
    \| (1+|\cdot|)^{-\beta}\hat f\|_{L^{q,\infty}(\R)} \les \|f\|_{L^p(\R)}.
  $$
\end{prop}
\begin{proof}
  Define $r$ via $\frac{1}{q}= \frac{1}{r}+\frac{1}{p'}$, so $r\in  [\frac{1}{\beta},\infty]$. 
  Then the claim follows from the Lorentz space version of  H\"older's Inequality and the
  Hausdorff-Young Inequality via
  \begin{align*}
    \| (1+|\cdot|)^{-\beta}\hat f\|_{L^{q,\infty}(\R)}
    \les \| (1+|\cdot|)^{-\beta}\|_{L^{r,\infty}(\R)} \|\hat f\|_{L^{p',\infty}(\R)}
    \les \|\hat f\|_{L^{p'}(\R)} 
    \les \|f\|_{L^p(\R)}.
  \end{align*}
\end{proof}

  We deduce some  mixed restricted weak-type bounds for $\mathcal
  R^*_{\alpha,\beta}$ for $\alpha\geq \frac{1}{q}$.  


  \begin{lem}\label{lem:weaktypebounds_largealpha}
      Assume $0\leq\beta\leq \alpha<\infty$ and $0<q<\infty,1\leq r\leq \infty$.
      \begin{itemize}
        \item[(i)] If $r<\infty$, $\alpha\geq \frac{1}{q}$ and $2\beta\geq \frac{2}{q}-\frac{1}{r'}$,  
         then  $\mathcal R_{\alpha,\beta}^*:L^{r,1}(\sph{1})\to
        L^{q,\infty}(\R;L^{q,\infty}(\R))$ is bounded. 
        \item[(ii)] If $r=\infty$, $\alpha\geq \frac{1}{q}$ and
      $2\beta>\frac{2}{q}-\frac{1}{r'}$, then $\mathcal R_{\alpha,\beta}^*:L^r(\sph{1})\to
      L^{q,\infty}(\R;L^{q,\infty}(\R))$ is bounded.
      \end{itemize}
  \end{lem}
  \begin{proof}
    We may w.l.o.g. assume $\beta\leq \frac{1}{q}$ and     
    that $F$ vanishes in the left half of the sphere. Set
    \begin{align*}
      F_\star(\phi)
       :=F(\cos(\phi),\sin(\phi)),\qquad 
       G_{x}(s) 
       := \sqrt{2\pi} F_\star(-\arcsin(s)) (1-s^2)^{-\frac{1}{2}} e^{ix(1-s^2)^{\frac{1}{2}}}
       \ind_{[-1,1]}(s).
    \end{align*}
    Then
   \begin{align*}
     (\mathcal R^*F)(x,y)
     &= \int_{\sph{1}} F(z) e^{i(x,y)\cdot (z_1,z_2)}\,d\sigma(z) \\
     &= \int_{-\frac{\pi}{2}}^{\frac{\pi}{2}} F_\star(\phi) e^{i(x,y)\cdot(\cos(\phi),\sin(\phi))}\,d\phi \\
     &= \int_{-1}^1 F_\star(\arcsin(s))(1-s^2)^{-\frac{1}{2}} e^{ix(1-s^2)^{\frac{1}{2}}} e^{isy} \,ds \\
     &= \frac{1}{\sqrt{2\pi}} \int_\R G_x(s)e^{-isy}  \,ds  \\
     &= \widehat{G_{x}}(y).
   \end{align*}
   Define the exponents $p,\mu$ via 
   $$
     \frac{1}{p}:=\frac{1}{q'}+\beta,\qquad 
     \frac{1}{\mu} :=\frac{1}{p}-\frac{1}{p'}.
   $$
   Our assumptions on $\beta$ imply $1\leq p\leq 2$ and $1\leq \mu\leq r$.
   From $\alpha\geq \frac{1}{q}$ and Proposition~\ref{prop:Pitttype}  we infer   
   \begin{align*}
     \|\mathcal R^*_{\alpha,\beta}F\|_{L^{q,\infty}(\R;L^{q,\infty}(\R))}
     &=  \Big\| (1+|x|)^{-\alpha} \cdot \| (1+|y|)^{-\beta}\hat G_{x}(y)\|_{L_y^{q,\infty}(\R)}
     \Big\|_{L_x^{q,\infty}(\R)}   \\
     &\les 
      \| (1+|x|)^{-\alpha}\|_{L_x^{q,\infty}(\R)} \cdot
     \sup_{x\in\R}  \| (1+|y|)^{-\beta}\hat G_{x}(y)\|_{L_y^{q,\infty}(\R)} \\
     &\les\sup_{x\in\R}   \|G_{x}\|_{L^p(\R)}  \\
     &\simeq  \|F_\star(-\arcsin(\cdot)) (1-(\cdot)^2)^{-\frac{1}{2}} \|_{L^p([-1,1])}   \\
     &\simeq    \|F_\star |\cos|^{\frac{1-p}{p}}  \|_{L^p([-\frac{\pi}{2},\frac{\pi}{2}])}           \\
     &\les   \|F_\star\|_{L^{\mu,p}([-\frac{\pi}{2},{\frac{\pi}{2}}])}
     \| |\cos|^{\frac{1-p}{p}} \|_{L^{\frac{p}{p-1},\infty}([0,{\frac{\pi}{2}}])}  \\
     &\les   \|F_\star\|_{L^{r,1}([-\frac{\pi}{2},{\frac{\pi}{2}}])} \\
     &\les   \|F\|_{L^{r,1}(\S)}.
   \end{align*}
   Here the third last estimate is justified if and only if $1\leq \mu<\infty$, which is automatically
   satisfied in the case $r<\infty$. So this proves (i).
   In the case $r=\infty$, an analogous estimate holds with 
   $\|F\|_{L^\infty(\S)}$ on the right hand side whenever $1\leq \mu<\infty$, which is exactly the
   case if $2\beta>\frac{2}{q}-1=\frac{2}{q}-\frac{1}{r'}$ is assumed. This gives (ii).
  \end{proof}

  We shall partially upgrade these estimates via interpolation. Before doing this,
  we extend the range of restricted weak-type estimates to parameters $\alpha\in (0,\frac{1}{q})$. This is
  done via Stein interpolation in the setting of quasi-Banach spaces as presented earlier. 
  In contrast to the previous result the range   will be $L^{q,\infty}(\R^2)$ instead of $L^{q,\infty}(\R;L^{q,\infty}(\R))$.

  \begin{lem}\label{lem:weaktypebounds_smallalpha}
      Assume $0\leq \beta\leq \alpha<\frac{1}{q}$ and $1\leq r\leq\infty, 0<q<\infty$. 
      \begin{itemize}
        \item[(i)] If $r<\infty$ then $\mathcal R_{\alpha,\beta}^*:L^{r,1}(\sph{1})\to
        L^{q,\infty}(\R^2)$ is bounded if  $\alpha+\beta>\frac{2}{q}-\frac{1}{2},\,\alpha+2\beta\geq
        \frac{3}{q}-\frac{1}{r'}$.
        \item[(ii)] If $r=\infty$ then $\mathcal R_{\alpha,\beta}^*:L^r(\sph{1})\to
        L^{q,\infty}(\R^2)$ is bounded if  $\alpha+\beta>\frac{2}{q}-\frac{1}{2},\,\alpha+2\beta>
        \frac{3}{q}-\frac{1}{r'}$.
      \end{itemize}
  \end{lem}
  \begin{proof}
    We interpolate the estimates from Lemma~\ref{lem:weaktypebounds_largealpha} with the unweighted ones from
    Theorem~\ref{thm:FeffermanZygmund} and concentrate on the case $r<\infty$ in (i). We show in the Appendix 
    (Proposition~\ref{prop:InterpolationArithI})
    that our assumptions on $\alpha,\beta,r,q$ allow to find $r_0\in [1,\infty],r_1\in [1,\infty), q_0\in
    [1,\infty],q_1\in (0,\infty)$ and $\theta\in (0,1)$ such that
   \begin{align*}
     &\frac{1}{q}=\frac{1-\theta}{q_0}+\frac{\theta}{q_1},\quad
     \frac{1}{r}=\frac{1-\theta}{r_0}+\frac{\theta}{r_1}\qquad \text{and} \\
     &\frac{1}{q_0} \leq \frac{1}{3r_0'},\quad\frac{1}{q_0}<\frac{1}{4},\quad
     \frac{\alpha}{\theta} = \frac{1}{q_1},\quad
     \frac{\beta}{\theta}\geq \frac{1}{q_1}-\frac{1}{2r_1'},\quad q_0\neq q_1.
   \end{align*}
   Then the family of linear operators
   $\mathcal E_s:= \mathcal R_{s\alpha/\theta,s\beta/\theta}^*$ has the following properties:
   \begin{itemize}
     \item[(a)]  $\mathcal E_s: L^{r_0}(\S)\to L^{q_0}(\R^2)$ is bounded for
     $\Re(s)=0$. \\
     This follows from Theorem~\ref{thm:FeffermanZygmund} and $\frac{1}{q_0}\leq \frac{1}{3r_0'},
     \frac{1}{q_0}<\frac{1}{4}$.
     \item[(b)] $\mathcal E_s: L^{r_1,1}(\S)\to L^{q_1,\infty}_x(\R;L^{q_1,\infty}_y(\R)) 
     $ is bounded for $\Re(s)=1$. \\
     This follows from Lemma~\ref{lem:weaktypebounds_largealpha}~(i) and $\frac{\alpha}{\theta}=
     \frac{1}{q_1},\; \frac{\beta}{\theta}\geq \frac{1}{q_1}-\frac{1}{2r_1'}$.
   \end{itemize}
   Stein interpolation (Theorem~\ref{thm:GrafMas} for  $\mathcal X:= L^\infty(\mathbb{S}^1)$)
   then implies the boundedness of $$
     \mathcal E_s: [L^{r_0}(\S),L^{r_1,1}(\S)]_\theta \to
     L^{q_0}(\R^2)^{1-\theta}\big(L^{q_1,\infty}(\R;L^{q_1,\infty}(\R))\big)^\theta
     \quad\text{whenever }\Re(s)=\theta.
   $$   
   Plugging in $s=\theta$ and using the embeddings   
   $$
     L^{r,1}(\S)\subset [L^{r_0}(\S),L^{r_1,1}(\S)]_\theta
     \quad\text{and}\quad   
     L^{q,\infty}(\R^2)\supset  L^{q_0}(\R^2)^{1-\theta}\big(L^{q_1,\infty}(\R;L^{q_1,\infty}(\R))\big)^\theta
   $$ 
   from Proposition~\ref{prop:embedding} (note $q_0\neq q_1$) 
   we deduce the boundedness of $\mathcal R_{\alpha,\beta}^* = \mathcal E_\theta: L^{r,1}(\S)\to
   L^{q,\infty}(\R^2)$ for the claimed range of exponents $q,r$. Part (ii) follows from (i) and
   $L^\infty(\S)\subset L^{\tilde r}(\S)$ for all $\tilde r\in [1,\infty)$.

  \end{proof}
  
  \begin{cor}\label{cor:mappingproperties}
    Assume $0\leq \beta\leq \alpha<\infty$ and $1<r,q<\infty$. Then
    $\mathcal R_{\alpha,\beta}^*:L^{r,s}(\sph{1})\to L^{q,s}(\R^2)$ is bounded for all $s\in [1,\infty]$ if
    one of the following conditions holds:
    \begin{itemize}
      \item[(i)] $\alpha>\frac{1}{q}$ and $2\beta\geq \frac{2}{q}-\frac{1}{r'}$,
      \item[(ii)] $\alpha=\frac{1}{q}$ and  $2\beta> \frac{2}{q}-\frac{1}{r'}$,
      \item[(iii)] $\alpha<\frac{1}{q}$ and $\alpha+\beta>\frac{2}{q}-\frac{1}{2}$ as well as
       $\alpha+2\beta\geq \frac{3}{q}-\frac{1}{r'}$.
    \end{itemize}
  \end{cor}
  \begin{proof}
    We first consider the range $\alpha>\frac{1}{q}$ where it suffices to prove
    the endpoint estimate where $2\beta= \frac{2}{q}-\frac{1}{r'}$. In view of $1<r,q<\infty$
    we may choose $\eps>0$ sufficiently small and define $q_0,q_1,r_0,r_1\in (1,\infty)$ via 
    $$
      \frac{1}{q_0}:=\frac{1}{q}-\eps,\quad
      \frac{1}{q_1}:=\frac{1}{q}+\eps,\quad
      \frac{1}{r_0}:=\frac{1}{r}+2\eps,\quad
      \frac{1}{r_1}:=\frac{1}{r}-2\eps.
    $$ 
    We then have $\alpha>\frac{1}{q_i}$ and $2\beta= \frac{2}{q_i}-\frac{1}{r_i'}$ for  $i=0,1$, so
    Lemma~\ref{lem:weaktypebounds_largealpha} implies that $\mathcal R_{\alpha,\beta}^*:L^{r_i,1}(\sph{1})\to
    L^{q_i,\infty}(\R;L^{q_i,\infty}(\R))$ is bounded. Hence, real interpolation yields
    the boundedness of 
    $$
      \mathcal R_{\alpha,\beta}^*:
      \big(L^{r_0,1}(\sph{1}),L^{r_1,1}(\sph{1})\big)_{\theta,s} 
      \to \big(L^{q_0,\infty}(\R;L^{q_0,\infty}(\R)),L^{q_1,\infty}(\R;L^{q_1,\infty}(\R))\big)_{\theta,s}
    $$ 
    for all $s\in [0,1]$. Since $q_0>q>q_1,r_0<r<r_1$ and
    $\frac{1}{q}=\frac{1-\theta}{q_0}+\frac{\theta}{q_1},\frac{1}{r}=\frac{1-\theta}{r_0}+\frac{\theta}{r_1}$
    for $\theta:=\frac{1}{2}$, real interpolation \cite[Theorem~3.11.8]{BL} 
    and Proposition~\ref{prop:embedding}~(i) give the boundedness of
     $\mathcal R_{\alpha,\beta}^*: L^{r,s}(\sph{1}) \to L^{q,s}(\R^2)$ for all $s\in [1,\infty]$, which
    implies claim (i).
    
    \medskip
    
   In order to prove (iii) assume $0<\alpha<\frac{1}{q}$ and choose $q_0,q_1,r_0,r_1\in (1,\infty)$ such that
   $$
      \frac{1}{q_0}:=\frac{1}{q}-\eps,\quad
      \frac{1}{q_1}:=\frac{1}{q}+\eps,\quad
      \frac{1}{r_0}:=\frac{1}{r}+3\eps,\quad
      \frac{1}{r_1}:=\frac{1}{r}-3\eps
    $$ 
   for some sufficiently small $\eps>0$. Then one may check $0<\alpha<\frac{1}{q_i}$, 
   $\alpha+\beta>\frac{2}{q_i}-\frac{1}{2}$ as well as
   $\alpha+2\beta\geq  \frac{3}{q_i}-\frac{1}{r_i'}$. As above, real interpolation yields the boundedness of
   $\mathcal R_{\alpha,\beta}^*:  L^{r,s}(\sph{1})\to L^{q,s}(\R^2)$ for all $s\in [1,\infty]$. 
   The case (ii) may be deduced from (iii).     
  \end{proof}
  
  \medskip

  \textbf{Proof of Theorem~\ref{thm:main} -- sufficient conditions:}
   We prove the sufficiency part of the Theorem assuming $0<q<\infty$ and w.l.o.g. $\alpha\geq
   \beta\geq 0$. So want have to show that $\mathcal R_{\alpha,\beta}^*:L^r(\mathbb{S}^1)\to
   L^q(\R^2)$ is bounded for $q,r$ as stated in the theorem.
   
   \smallskip \textit{Endpoint estimates (iii),(iv):} Claim (iii) is about $\alpha>\frac{1}{q},
   2\beta=\frac{2}{q}-\frac{1}{r'}$ and $1<r\leq q$. Here, the bound results from
    Corollary~\ref{cor:mappingproperties}~(i) thanks to the embeddings $L^r(\sph{1})\subset
    L^{r,s}(\sph{1})$ and $L^{q,s}(\R^2)\subset L^q(\R^2)$ where $s\in [r,q]$.  
    Analogously, Corollary~\ref{cor:mappingproperties}~(iii) yields 
    the endpoint estimate  (iv) dealing with the case
    $\alpha<\frac{1}{q},\alpha+2\beta=\frac{3}{q}-\frac{1}{r'}$ and $1<r\leq q$.
    
    \smallskip  \textit{Nonendpoint estimates (i),(ii):} The nonendpoint
    estimates concern the parameter region   
    $$
      \alpha\geq \frac{1}{q},\;2\beta>\frac{2}{q}-\frac{1}{r'}
      \quad\text{or}\quad
      \alpha< \frac{1}{q},\;\alpha+2\beta>\frac{3}{q}-\frac{1}{r'}.
     $$  
    In the case $1<r\leq \infty$, the boundedness of $R^*_{\alpha,\beta}:L^r(\mathbb{S}^1)\to
    L^q(\R^2)$ is a consequence of Corollary~\ref{cor:mappingproperties}. Indeed, in that case we can
    find a sufficiently small $\eps>0$ such that the Corollary implies the boundedness 
    from $L^{r_\eps,q}(\mathbb{S}^1)$ to $L^q(\R^2)$ where $\frac{1}{r_\eps} := \frac{1}{r}+\eps$. 
    So the embedding $L^{r}(\mathbb{S}^1)\hookrightarrow L^{r_\eps,q}(\mathbb{S}^1)$ yields the claim.  
    So it remains to deal with the nonendpoint estimates for  $r=1$. Since (ii) is void, we have to prove the
    boundedness for $\alpha\geq \frac{1}{q},\beta>\frac{1}{q}$. 
    Lemma~\ref{lem:weaktypebounds_largealpha}~(i) implies the boundedness $L^1(\mathbb{S}^1)\to
    L^{q-\eps,\infty}(\R;L^{q-\eps,\infty}(\R))$ for sufficiently small $\eps>0$. On the other hand, the
    boundedness $L^1(\mathbb{S}^1)\to L^\infty(\R^2)$ is trivial. So interpolation and
    Proposition~\ref{prop:embedding}~(ii) implies the boundedness  of
    $R^*_{\alpha,\beta}: L^1(\mathbb{S}^1)\to L^{q}(\R^2)$, which is all we had to show.
%
     \qed

\section{Proof of Theorem~\ref{thm:main} -- Necessary conditions} \label{sec:Counterexamples}

  We now show that our result is sharp in the scale of Lebesgue spaces. This is mainly achieved by a
  detailed analysis of the constant density $F\equiv 1$, which turns out to be responsible for the necessary
  condition $\alpha+\beta>\frac{2}{q}-\frac{1}{2}$, and Knapp-type counterexamples.  
  In the following we shall always assume at least 
  $$
    0\leq\beta\leq \alpha<\infty, \qquad
    1\leq r\leq \infty, \qquad
    0<q<\infty.
  $$   
  The necessity of the  conditions in Theorem~\ref{thm:main} follows from four Lemmas
  in this section. We shall prove:
   \begin{itemize}
     \item[(i)] The condition $\alpha+\beta>\frac{2}{q}-\frac{1}{2}$ is necessary by Lemma~\ref{lem:NecI}.
     \item[(ii)] Endpoint estimates for $\alpha>\frac{1}{q}$ can only hold for
     $q\geq r>1$ by Lemma~\ref{lem:NecII} and \ref{lem:NecIII}.
     \item[(iii)]  Endpoint estimates for $\alpha=\frac{1}{q}$ are
     impossible by Lemma~\ref{lem:NecII}. 
     \item[(iv)] Endpoint estimates for $\alpha<\frac{1}{q}$  
     can only hold for $q\geq r$ by Lemma~\ref{lem:NecIV}.       
   \end{itemize}
   From (ii)-(iv) we infer that endpoint estimates can only hold if
   $1<r\leq q$ and $\alpha\neq \frac{1}{q}$. Note that   
   the endpoint estimates for $\alpha<\frac{1}{q}$ do not occur for $r=1$.
   Combining this with Remark~\ref{rem:general}(a) we obtain the necessity
   of all sufficient conditions stated in Theorem~\ref{thm:main}.

  \medskip

  The first counterexample deals with the constant density on the unit sphere. The Fourier transform of the
  surface measure  of the unit sphere in $\R^2$ is given by
  $\widehat{\sigma}(x) = c J_0(|x|)$ where $c$ is some absolute constant and 
  $J_0$ denotes a Bessel function of the second kind.
  The properties of such Bessel functions are well-known; we shall only need the lower bound
  \begin{equation} \label{eq:estJ0}
    |J_0(r)| \gtrsim \sum_{j=1}^\infty j^{-\frac{1}{2}} \ind_{[z_j-\delta,z_j+\delta]}(r)
  \end{equation}
  where $\delta>0$ is some small fixed number and $\{z_j:j\in\N\}$ denotes the set of positive local extrema
  of $J_0$. It is known that $c\,j\leq z_j\leq C\, j$ and that $z_j-z_{j-1}$ converges to a positive constant
  as $j\to\infty$. So $\delta>0$ can be assumed so small that the intervals $[z_j-\delta,z_j+\delta]$ are
  mutually disjoint.

    \begin{lem}  \label{lem:NecI}
      Assume $0\leq\beta\leq \alpha<\infty, 1\leq r\leq \infty,0<q<\infty$. If the operator $\mathcal
      R^*_{\alpha,\beta}:L^r(\sph{1})\to L^q(\R^2)$  is bounded, then $\alpha+\beta>\frac{2}{q}-\frac{1}{2}$.
    \end{lem}
    \begin{proof}
      For $F:= 1\in L^r(\sph{1})$ we have
    \begin{align*}
      \|\mathcal R_{\alpha,\beta}^*F\|_{L^q(\R^2)}^q
      &= \| J_0(|(x,y)|)(1+|x|)^{-\alpha} (1+|y|)^{-\beta}\|_{L^q(\R^2)}^q \\
      &\stackrel{\eqref{eq:estJ0}}\gtrsim \sum_{j=1}^\infty j^{-\frac{q}{2}}\| (1+|x|)^{-\alpha}
      (1+|y|)^{-\beta} \ind_{[z_j-\delta,z_j+\delta]}(|(x,y)|)\|_{L^q(\R^2)}^q  \\
      &\gtrsim
      \sum_{j=1}^\infty j^{-\frac{q}{2}} \int_{z_j/4}^{z_j/2}  (1+|x|)^{-q\alpha}
      \left(\int_{\sqrt{(z_j-\delta)^2-|x|^2}}^{\sqrt{(z_j+\delta)^2-|x|^2}} (1+|y|)^{-q\beta} \,dy \right)
      \,dx \\
      &\gtrsim
      \sum_{j=1}^\infty j^{-\frac{q}{2}} \int_{z_j/4}^{z_j/2}  (1+|x|)^{-q\alpha}
      \cdot (1+|z_j|)^{-q\beta}  
      \left(\int_{\sqrt{(z_j-\delta)^2-|x|^2}}^{\sqrt{(z_j+\delta)^2-|x|^2}} 1 \,dy \right)
        \,dx \\
      &\gtrsim
      \sum_{j=1}^\infty j^{-\frac{q}{2}} \int_{z_j/4}^{z_j/2}  (1+|x|)^{-q\alpha}
      \cdot \delta (1+|z_j|)^{-q\beta}  \,dx \\
      &\gtrsim
      \sum_{j=1}^\infty j^{-\frac{q}{2}}  z_j (1+|z_j|)^{-q\alpha-q\beta}    \\
      &\gtrsim
      \sum_{j=1}^\infty j^{1-q(\frac{1}{2}+\alpha+\beta)}.
    \end{align*}
    So $\mathcal R_{\alpha,\beta}^*F\in L^q(\R^2)$ implies $1-q(\frac{1}{2}+\alpha+\beta)<-1$, which is
    equivalent to $\alpha+\beta>\frac{2}{q}-\frac{1}{2}$.
 \end{proof}
 
 Next we discuss the Knapp example following the ideas from~\cite{ManDOS}.   

 \begin{lem}  \label{lem:NecII}
   Assume  $0\leq\beta\leq \alpha<\infty, 1\leq r\leq \infty,0<q<\infty$.  Then the following conditions are
   necessary for $\mathcal R^*_{\alpha,\beta}:L^r(\S)\to L^q(\R^2)$ to be bounded:
   \begin{itemize}
     \item[(i)] if $\alpha>\frac{1}{q}$, then $2\beta\geq \frac{2}{q}-\frac{1}{r'},r>1$ or
     $2\beta>\frac{2}{q},r=1$, 
     \item[(ii)] if $\alpha=\frac{1}{q}$, then $2\beta>\frac{2}{q}-\frac{1}{r'}$,
     \item[(iii)] if $\alpha<\frac{1}{q}$, then $\alpha+2\beta\geq \frac{3}{q}-\frac{1}{r'}$.  
   \end{itemize}   
 \end{lem}
 \begin{proof}
   For $\delta>0$ consider  $F:= \ind_{\mathcal C_\delta}$ where $\mathcal C_\delta := \{
   \xi=(\xi_1,\xi_2)\in\S:
   |\xi_1|<\delta\}$ denotes a spherical cap. Then, for $1\leq r<\infty$, 
   $$
     \|F\|_{L^r(\S)}
     = \left(\int_{\mathcal C_\delta} \,d\sigma(\xi)\right)^{\frac{1}{r}} 
     = \left(\int_0^{2\pi} \ind_{|\sin(\phi)|<\delta} \,d\phi\right)^{\frac{1}{r}}
     \simeq \delta^{\frac{1}{r}} 
   $$
   and the same bound holds for $r=\infty$.
   To estimate $\|\mathcal R^*_{\alpha,\beta}F\|_{L^q(\R^2)}$ from below, define
   $$
     E_j:= \left\{(x,y)\in\R^2 : 0<|x|\leq c\delta^{-1}, \frac{2\pi j- c}{\sqrt{1-\delta^2}}\leq y\leq 2\pi
     j+c\right\} 
   $$ 
   for some small enough constant $c>0$, say $c:=\frac{\pi}{8}$. We then have for $(x,y)\in E_j$,
   \begin{align*}
     |(\mathcal R^*F)(x,y)|  
     &= \left| \int_{\mathcal C_\delta} e^{i\xi\cdot(x,y)}\,d\sigma(\xi) \right| \\
     &= \left| \int_0^{2\pi} \ind_{|\sin(\phi)|<\delta} e^{i(\sin(\phi) x + \cos(\phi)y-2\pi j)}\,d\phi
     \right|
     \\
     &\geq   \int_0^{2\pi} \ind_{|\sin(\phi)|<\delta} \cos\big( \sin(\phi) x + \cos(\phi)y-2\pi j  
     \big)\,d\phi
   \end{align*}
   The argument of cosine has absolute value at most $\frac{\pi}{4}$ by choice of $c>0$ and $(x,y)\in E_j$.
   Indeed,
   \begin{align*}
     \sin(\phi) x + \cos(\phi)y-2\pi j 
     &\geq - \delta|x| + \sqrt{1-\delta^2}y - 2\pi j  
     \geq  - \delta\cdot c\delta^{-1} -c 
     = -2c 
     = - \frac{\pi}{4}, \\ 
     \sin(\phi) x + \cos(\phi)y-2\pi j
     &\leq  \delta |x| + y-2\pi j 
     \leq   \delta\cdot c\delta^{-1}+c 
     = 2c
     = \frac{\pi}{4}.
   \end{align*}
   So we obtain
   \begin{align*}
     |(\mathcal R^*F)(x,y)|  
     \geq   \int_0^{2\pi} \ind_{|\sin(\phi)|<\delta} \cos\big(\frac{\pi}{4}\big)\,d\phi 
     \gtrsim \delta \quad\text{for }(x,y)\in E_j, j\in\N.
   \end{align*} 
   
   Moreover, the sets $E_j$ are mutually disjoint and satisfy
   $$
     2\pi j+c - \frac{2\pi j-c}{\sqrt{1-\delta^2}}\geq c 
     \qquad\text{for } j=1,\ldots,\lfloor c_0\delta^{-2}\rfloor
   $$ 
   provided that $c_0>0$ is small enough. This implies  
   \begin{align*}
     \frac{\|\mathcal R^*_{\alpha,\beta}F\|_{L^q(\R^2)}}{\|F\|_{L^r(\S)}}
     &\gtrsim \delta^{-\frac{1}{r}} \| (\mathcal R^*F)(x,y) (1+|x|)^{-\alpha}(1+|y|)^{-\beta}\|_{L^q(\R^2)} \\
     &\gtrsim \delta^{-\frac{1}{r}}\cdot  \left(\sum_{j=1}^{\lfloor c_0\delta^{-2}\rfloor} \| \ind_{E_j}
     \delta (1+|x|)^{-\alpha}(1+|y|)^{-\beta}\|_{L^q(\R^2)}^q \right)^{\frac{1}{q}} \\ 
     &\simeq \delta^{-\frac{1}{r}} \left(\sum_{j=1}^{\lfloor c_0 \delta^{-2}\rfloor} \delta^q 
     \int_0^{c\delta^{-1}}  (1+|x|)^{-\alpha q}\,dx \int_{\frac{2\pi j-c}{\sqrt{1-\delta^2}}}^{2\pi j+c}
     (1+|y|)^{-\beta q}\,dy  \right)^{\frac{1}{q}} \\
     &\simeq \delta^{\frac{1}{r'}} \left(
     \int_0^{c\delta^{-1}}  (1+|x|)^{-\alpha q}  \,dx    \cdot
     \sum_{j=1}^{\lfloor c_0\delta^{-2}\rfloor} j^{-\beta q} 
     \right)^{\frac{1}{q}} \\
     &\gtrsim  \delta^{\frac{1}{r'}} \left( 
     \begin{cases}
       1 &, \alpha >\frac{1}{q}\\
       |\log(\delta)| &, \alpha = \frac{1}{q}\\
       \delta^{-1+\alpha q} &, \alpha<\frac{1}{q}
     \end{cases}
     \cdot
       \begin{cases}
       1 &, \beta >\frac{1}{q}\\
       |\log(\delta)| &, \beta = \frac{1}{q}\\
       \delta^{-2+2\beta q} &, \beta<\frac{1}{q}
     \end{cases}       
     \right)^{\frac{1}{q}}.    
   \end{align*}
   So a case distinction gives the result. 
 \end{proof}

 We start with proving the necessity of $1<r\leq q$ in the case of endpoint estimates in the parameter
 range $\alpha>\frac{1}{q}$.
 
 \begin{lem}   \label{lem:NecIII}
   Assume $0\leq\beta\leq \alpha<\infty$ and $0<q<\infty,1<r\leq \infty$ where $\alpha>\frac{1}{q}$,
   $\beta=\frac{1}{q}-\frac{1}{2r'}$.  Then $\mathcal R^*_{\alpha,\beta}:L^r(\S)\to L^q(\R^2)$ can only be
   bounded if $q\geq r$.
 \end{lem} 
 \begin{proof}
   By the monotonicity considerations from Remark~\ref{rem:general}~(a) with respect to $q$, it
   suffices to derive the necessity of $q\geq r$ assuming a priori $q> 1$.      
   We actually prove the equivalent statement that the  dual bound 
   $$
      \|\hat f\|_{L^{r'}(\S)}  \lesssim \|(1+|x|)^{\alpha}(1+|y|)^\beta f\|_{L^{q'}(\R^2)}
  $$
  can only hold for $q\geq r$. We consider functions of the form
  $$
    f_\eps(x,y):=\chi(x)(1+|y|)^{-\beta-\frac{1+\eps}{q'}} e^{iy}
  $$ 
  where $\eps>0$ is small and $\chi$ is a nontrivial nonnegative Schwartz function on $\R$. Then 
  \begin{align*}
     \|(1+|x|)^{\alpha}(1+|y|)^\beta f_\eps\|_{L^{q'}(\R^2)}
     \les  \left( \int_0^\infty (1+y)^{-1-\eps}  \,dy \right)^{\frac{1}{q'}} 
     \les \eps^{-\frac{1}{q'}}.
  \end{align*}
  Since $r>1$, the exponents 
  $\beta+\frac{1+\eps}{q'}= 1-\frac{1}{2r'}+\frac{\eps}{q'}$ are contained in
  $(0,1)$ provided that $\eps>0$ is sufficiently small. So Proposition~\ref{prop:oscint} in the Appendix
  applies and yields as $(\xi_1,\xi_2)\to (0,1)$    
   \begin{align*} 
     |\hat f_\eps(\xi)|
     &= \left| \int_{\R} \chi(x)e^{-i\xi_1\cdot x}\,dx
     \cdot \int_\R (1+|y|)^{-1+\frac{1}{2r'}-\frac{\eps}{q'}}  e^{-i(\xi_2-1)\cdot y}\,dy \right| \\
     &= \left| \int_{\R} \chi(x)e^{-i\xi_1\cdot x}\,dx \right|
     \cdot 2\left|\int_0^\infty (1+r)^{-1+\frac{1}{2r'}-\frac{\eps}{q'}}  \cos((\xi_2-1)r)\,dr \right| 
     \\
     &\gtrsim  \left(\int_\R \chi(x)\,dx \right) |\xi_2-1|^{-\frac{1}{2r'}+\frac{\eps}{q'}} \\
     &\gtrsim  |\xi_2-1|^{-\frac{1}{2r'}+\frac{\eps}{q'}}
   \end{align*}
   uniformly with respect to small $\eps>0$. So the assumptions implies, in view of $r'<\infty$, for some
   small enough $c>0$
   \begin{align*}
     \eps^{-\frac{1}{q'}}
     &\gtrsim \|\hat f_\eps\|_{L^{r'}(\S)}   
     \;=\; \left( \int_0^{2\pi} |\hat f_\eps(\sin(t),\cos(t))|^{r'}\,dt \right)^{\frac{1}{r'}}\\
     &\gtrsim \left( \int_0^{c}  
     |\cos(t)-1|^{r'(-\frac{1}{2r'}+\frac{\eps}{q'})}  \,dt \right)^{\frac{1}{r'}} 
     \;\gtrsim\; \left( \int_0^{c} t^{-1+\frac{2r'\eps}{q'}} \,dt \right)^{\frac{1}{r'}} \\  
     &\simeq  \eps^{-\frac{1}{r'}}. 
   \end{align*}
   From this we infer $r'\geq q'$, hence $q\geq r$. 
 \end{proof}
 
 Next we extend this result to the range $\alpha<\frac{1}{q}$ and prove the necessity of $q\geq r$ in the
 special case $q=2$.
 
 \begin{lem} \label{lem:OptimalL2Result} 
   Assume  $0\leq \beta\leq \alpha<\frac{1}{2},1\leq r\leq \infty$ and 
   $\alpha+\beta>\frac{1}{2}$ as well as $\alpha+2\beta = \frac{3}{2}-\frac{1}{r'}$. Then
   $\mathcal R^*_{\alpha,\beta}:L^r(\S)\to L^2(\R^2)$ can only be bounded if $2\geq r$.
 \end{lem}
 \begin{proof}
   First we observe that the assumptions imply $1<r<\infty$. Indeed, if $r=1$, then $3\alpha\geq
   \alpha+2\beta=\frac{3}{2}>3\alpha$ and if $r=\infty$, then $\alpha+2\beta\geq
   \alpha+\beta>\frac{1}{2}=\alpha+2\beta$. Both statements are absurd, so $1<r<\infty$. Moreover, we
   necessarily have $\beta>0$ in view of $\alpha<\frac{1}{2}<\alpha+\beta$.  To prove the claim 
   set 
   $$
     F_\delta(\sin\phi,\cos\phi):= \phi^{-\mu} \ind_{[0,\delta]}(\phi)\quad\text{for } \phi\in [0,2\pi]
   $$
   and $\mu\in (0,\frac{1}{r})$ that we will send to $\frac{1}{r}<1$ from below. The parameter $\delta>0$ will be chosen
   sufficiently small but fixed. Then 
   $$
     \|F_\delta\|_{L^r(\S)}^2
      \simeq \left(\int_0^\delta \phi^{-\mu r}\,d\phi\right)^{\frac{2}{r}} 
      = \left( \frac{1}{1-\mu r} \delta^{1-\mu r}\right)^{\frac{2}{r}}
      \simeq (1-\mu r)^{-\frac{2}{r}}.
   $$  
   On the other hand, we have  
   \begin{align*}
     &\;\|\mathcal R_{\alpha,\beta}^* F_\delta\|_{L^2(\R^2)}^2 \\
     &= \int_\R \int_\R  (1+|x|)^{-2\alpha} (1+|y|)^{-2\beta} 
     \left| \int_{\S} F_\delta(\xi) e^{i\xi\cdot
     (x,y)}\,\,d\sigma(\xi) \right|^2 \,dy\,dx \\
     &=  \int_0^\delta \int_0^\delta \phi^{-\mu}\varphi^{-\mu} 
     \left(\int_\R (1+|x|)^{-2\alpha} e^{i\lambda_{\phi,\varphi}x} \,dx\right)
     \left(\int_\R (1+|y|)^{-2\beta} e^{-i\mu_{\phi,\varphi}y} \,dy\right)
     \,d\varphi \,d\phi\\      
     &=  4\int_0^\delta \int_0^\delta \phi^{-\mu}\varphi^{-\mu} 
     \left(\int_0^\infty (1+r)^{-2\alpha} \cos(\lambda_{\phi,\varphi}r) \,dr\right)
     \left(\int_0^\infty (1+r)^{-2\beta} \cos(\mu_{\phi,\varphi}r) \,dr\right)
     \,d\varphi \,d\phi\\       
     &=  8\int_0^\delta \phi^{-\mu} \int_0^\phi \varphi^{-\mu} 
      \left(\int_0^\infty (1+r)^{-2\alpha} \cos(\lambda_{\phi,\varphi}r) \,dr\right)
     \left(\int_0^\infty (1+r)^{-2\beta} \cos(\mu_{\phi,\varphi}r) \,dr\right)
     \,d\varphi \,d\phi
   \end{align*}
   where $\lambda_{\phi,\varphi} = \sin(\phi)-\sin(\varphi)$ and $\mu_{\phi,\varphi} =
   \cos(\varphi)-\cos(\phi)$. We observe  $\lambda_{\phi,\varphi},\mu_{\phi,\varphi}\to 0^+$
   where $0\leq \varphi\leq \phi\leq \delta \to 0^+$. Hence,
   for small $\delta>0$ we obtain from Proposition~\ref{prop:oscint}, in view of $0<2\beta\leq 2\alpha<1$,    
   \begin{align*}
       \|\mathcal R_{\alpha,\beta}^* F_\delta\|_{L^2(\R^2)}^2 
     &\gtrsim  \int_0^\delta  \phi^{-\mu} \int_0^\phi \varphi^{-\mu}  
     \lambda_{\phi,\varphi}^{2\alpha-1}
     \mu_{\phi,\varphi}^{2\beta-1} \,d\varphi \,d\phi \\
     &=  \int_0^\delta \phi^{-\mu} \int_0^\phi \varphi^{-\mu} 
     (\sin(\phi)-\sin(\varphi))^{2\alpha-1}
     (\cos(\varphi)-\cos(\phi))^{2\beta-1} \,d\varphi \,d\phi. 
   \end{align*}
   From the Mean Value Theorem and $\alpha+2\beta=\frac{3}{2}-\frac{1}{r'}$ we get
   \begin{align*}
     \|\mathcal R_{\alpha,\beta}^* F_\delta\|_{L^2(\R^2)}^2
     &\gtrsim  \int_0^\delta \phi^{-\mu} \int_0^\phi \varphi^{-\mu} 
     (\phi-\varphi)^{2\alpha-1}
     (\phi(\phi-\varphi))^{2\beta-1}\,d\varphi \,d\phi \\
     &=  \int_0^\delta \phi^{2\beta-1-\mu} \int_0^\phi  
     (\phi-\varphi)^{2\alpha+2\beta-2}\varphi^{-\mu} \,d\varphi \,d\phi \\
     &=  \int_0^\delta \phi^{2\beta-1-\mu} \int_0^1 \phi  
     (\phi-s\phi)^{2\alpha+2\beta-2}(s\phi)^{-\mu} \,ds \,d\phi \\
     &=  \int_0^\delta \phi^{2\alpha+4\beta-2-2\mu}  \,d\phi \cdot 
     \int_0^1 (1-s)^{2\alpha+2\beta-2} s^{-\mu}\,ds \\
     &\gtrsim  \int_0^\delta \phi^{-1+\frac{2}{r}-2\mu}  \,d\phi \\
     &\simeq  (\frac{2}{r}-2\mu)^{-1}. 
   \end{align*}
   Here we used $\mu<\frac{1}{r}<1$. 
   Choose $\mu=\frac{1}{r}-\eps$ and send $\eps\to 0^+$.  Then
   \begin{align*}
     \frac{\|\mathcal R_{\alpha,\beta}^* F_\delta\|_{L^2(\R^2)}^2}{\|F_\delta\|_{L^r(\S)}^2} 
     \gtrsim \frac{(1-\mu r)^{2/r}}{\frac{2}{r}-2\mu}
     \gtrsim \frac{(\eps r)^{2/r}}{2\eps}
     \simeq \eps^{\frac{2}{r}-1},  
   \end{align*}
   so $r\leq 2$ is necessary for the boundedness of $\mathcal R_{\alpha,\beta}^*:L^r(\S)\to L^2(\R^2)$.
 \end{proof}
  
  We remark that the rather explicit expression for $\|\mathcal R_{\alpha,\beta}^* F\|_{L^2(\R^2)}$ from
  Lemma~\ref{lem:OptimalL2Result} may be used to give an alternative proof of $\|\mathcal
  R_{\alpha,\beta}^* F\|_{L^2(\R^2)}\les \|F\|_{L^r(\S)}$ under optimal conditions on $\alpha,\beta,r$. 
  In a very similar context this has been done in~\cite[Proposition~3.1]{ManDOS}. To conclude we now know 
  that $q\geq r$ is necessary for endpoint estimates in the special case $q=2$. We use this fact to derive the
  necessity of $q\geq r$ for all $q\in (0,\infty)$ using interpolation.

 \begin{lem} \label{lem:NecIV}
   Assume $0\leq \beta\leq \alpha<\frac{1}{q}$ and  $1<r\leq \infty,0<q\leq \infty$
   as well as $\alpha+\beta>\frac{2}{q}-\frac{1}{2}, \alpha+2\beta = \frac{3}{q}-\frac{1}{r'}$.  Then
   $\mathcal R^*_{\alpha,\beta}:L^r(\S)\to L^q(\R^2)$ can only be bounded if $q\geq r$. 
 \end{lem}
 \begin{proof}
   The case $q=2$ is already covered by Lemma~\ref{lem:OptimalL2Result}. To cover $q\in (0,2)$ and $q\in
   (2,\infty)$ we argue by contradicition. Assume first the claim is false for some $q>2$, i.e., 
   that there are $2<q_1<r_1$ and $\mathcal R^*_{\alpha_1,\beta_1}$ is bounded with
   $0\leq \beta_1\leq \alpha_1<\frac{1}{q_1}$ satisfying 
   $\alpha_1+\beta_1>\frac{2}{q_1}-\frac{1}{2}$ and
   $\alpha_1+2\beta_1=\frac{3}{q_1}-\frac{1}{r_1'}$.    
   Then choose $r_0=q_0\in  (1,2)$ and $\alpha_0\geq \beta_0>0$ such that
   \begin{equation}\label{eq:conditions}
     0< \alpha_0<\frac{1}{q_0},\quad 
     \alpha_0+\beta_0>\frac{2}{q_0}-\frac{1}{2},\quad
     \alpha_0+2\beta_0=\frac{3}{q_0}-\frac{1}{r_0'}.
   \end{equation}
   For instance, $\alpha_0:=\beta_0:=\frac{1}{3}(\frac{4}{q_0}-1)$ and $r_0:=q_0\in (1,2)$. 
   Theorem~\ref{thm:main}~(iii)   then implies that $\mathcal
   R^*_{\alpha_\theta,\beta_\theta}:L^{r_\theta}(\S)\to L^{q_\theta}(\R^2)$ is bounded for $\theta\in\{0,1\}$ 
   and  Stein interpolation (Theorem~\ref{thm:GrafMas}) allows to deduce that the same is true for $\theta\in (0,1)$.
   Here, $$
     \alpha_\theta = (1-\theta)\alpha_0+\theta\alpha_1,\quad
     \beta_\theta = (1-\theta)\beta_0+\theta\beta_1,\quad
     \frac{1}{r_\theta} = \frac{1-\theta}{r_0}+\frac{\theta}{r_1},\quad
     \frac{1}{q_\theta} = \frac{1-\theta}{q_0}+\frac{\theta}{q_1}.
   $$  
   This, however, contradicts our optimal result about the case $q=2$. To see why choose $\theta\in (0,1)$
   such that $q_\theta=2$, which is possible due to $q_0<2<q_1$. Then one checks 
   $$
     0<\beta_\theta\leq\alpha_\theta<\frac{1}{2},\quad
     \alpha_\theta+\beta_\theta>\frac{1}{2},\quad
     \alpha_\theta+2\beta_\theta = \frac{3}{2}-\frac{1}{r_\theta'}
     \quad\text{and}\quad
     \frac{1}{r_\theta}<\frac{1}{q_\theta} = \frac{1}{2}.
   $$
   But Lemma~\ref{lem:OptimalL2Result} implies that 
   $\mathcal R^*_{\alpha_\theta,\beta_\theta}:L^{r_\theta}(\S)\to L^2(\R^2)=L^{q_\theta}(\R^2)$ is not 
   bounded under these conditions, so the assumption was false and the claim is proved for $q>2$.  
   
   \medskip 
   
   In order to cover the case $0<q<2$ one similarly assumes for contradiction that 
   the estimate holds for some exponents $r_1>q_1$ with $q_1\in (0,2)$ and $\alpha_1,\beta_1$ as above. 
   Then one chooses $r_0,q_0\in (2,\infty)$ and $\alpha_0,\beta_0$ satisfying~\eqref{eq:conditions}, say
   $\alpha_0:=\beta_0:=\frac{1}{3}(\frac{4}{q_0}-1)$ and $r_0:=q_0\in (2,4)$. The same interpolation
   argument allows to derive a contradiction. After all, the assumption was false and the claim is proved.
 \end{proof}

 \section{Proof of Theorem~\ref{thm:main2} -- Sufficient conditions}
  
 In this section we determine the mapping properties of $\mathfrak R^*_\gamma:L^r(\S)\to L^q(\R^2)$. 
 We introduce the Banach space $(Z_q,\|\cdot\|_{Z_q})$ as follows:
 \begin{align*}
   Z_q&:= L^{q,\infty}_x(\R;L^{q,\infty}_y(\R))+L^{q,\infty}_y(\R;L^{q,\infty}_x(\R)), \\
   \|z\|_{Z_q} &:=  \inf\left\{\|z_1\|_{L^{q,\infty}_x(\R;L^{q,\infty}_y(\R))}+
   \|z_2\|_{L^{q,\infty}_y(\R;L^{q,\infty}_x(\R))}: z=z_1+z_2 \in Z_q\right\}.
 \end{align*}
 The counterpart of Lemma~\ref{lem:weaktypebounds_largealpha} reads as follows. 
 
  \begin{lem}\label{lem:weaktypebounds_radial}
      Assume $0\leq\gamma<\infty$ and $0<q\leq \infty,1\leq r\leq \infty$. Then
       $\mathfrak R_{\gamma}^*:L^r(\sph{1})\to Z_q$ is bounded provided that 
        $$
          \gamma \geq \max\left\{ \frac{1}{q}, \frac{2}{q}-\frac{1}{r'},\frac{2}{q}-\frac{1}{2}\right\}.
        $$
  \end{lem}
  \begin{proof}
    We first assume that $F:\S\to\R$ is supported away from the north and south pole, i.e.,
    $(0,1),(0,-1)\notin \supp(F)$. Then set $F_\star(\phi):=F(\cos(\phi),\sin(\phi))$. As in the
    proof of Lemma~\ref{lem:weaktypebounds_largealpha}, with identical notation, we have
   \begin{align*}
     (\mathcal R^*F)(x,y)
     = \int_{-1}^1 F_\star(\arcsin(s))(1-s^2)^{-\frac{1}{2}} e^{ix(1-s^2)^{\frac{1}{2}}} e^{isy} \,ds
      = \widehat{G_{x}}(y).
   \end{align*}
   Our assumption on the support of $F$ makes sure that the factor $(1-s^2)^{-1/2}$ may be treated as 
   a bounded function, set $H_x(z):= (1+|x|)^{-1}G_x( (1+|x|)^{-1}z)$. By assumption on $\gamma,q,r$ we may
   choose $p\in [1,\infty]$ such that  
   $$
     0\leq \frac{1}{p}-\frac{1}{q'}\leq \gamma,
     \quad \max\left\{\frac{1}{2},\frac{1}{r}\right\}\leq \frac{1}{p}\leq 1,
     \quad \frac{1}{q}-\gamma-1+\frac{1}{p}\leq -\frac{1}{q}.
   $$ 
   Then Proposition~\ref{prop:Pitttype} yields
   \begin{align*}
     \|\mathfrak R^*_{\gamma}F\|_{Z_q}
     &\leq \|\mathfrak R^*_{\gamma}F\|_{L^{q,\infty}_x(\R;L_y^{q,\infty}(\R))} \\
     &=  \Big\|  \| (1+|x|+|y|)^{-\gamma}\hat G_{x}(y)\|_{L_y^{q,\infty}(\R)}
     \Big\|_{L_x^{q,\infty}(\R)}   \\
     &=  \Big\|  \| (1+|x|+|y|)^{-\gamma}\hat H_{x}\big((1+|x|)^{-1} y\big)\|_{L_y^{q,\infty}(\R)}
     \Big\|_{L_x^{q,\infty}(\R)}   \\
     &=  \Big\|  (1+|x|)^{\frac{1}{q}-\gamma} \| (1+|z|)^{-\gamma}\hat H_{x}(z)\|_{L_z^{q,\infty}(\R)}
     \Big\|_{L_x^{q,\infty}(\R)}   \\
     &\les  \Big\|  (1+|x|)^{\frac{1}{q}-\gamma} \|H_x\|_{L^p(\R)} \Big\|_{L_x^{q,\infty}(\R)}   \\
     &\les    
     \Big\|  (1+|x|)^{\frac{1}{q}-\gamma-1+\frac{1}{p}}   \Big\|_{L_x^{q,\infty}(\R)}  
     \cdot \sup_{x\in\R} \|G_x\|_{L^p(\R)}  \\
     &\les   \|F_\star\|_{L^p([-1,1])}  \\
     &\les   \|F\|_{L^p(\S)} \\
     &\les   \|F\|_{L^r(\S)}.  
   \end{align*}
   The reasoning for $F\in L^r(\S)$ with $(1,0),(-1,0)\notin \supp(F)$ is analogous since it suffices to
   repeat the analysis with the roles of $x$ and $y$ interchanged. Since we can write
   $F=F_1+F_2$ and thus $\mathfrak R_\gamma^* F= \mathfrak R_\gamma^* F_1+\mathfrak R_\gamma^* F_2$    
   with $(0,\pm 1)\notin \supp(F_1),(\pm 1,0)\notin \supp(F_2)$, the claim follows.
  \end{proof}
  
  As in the proof of our first result, this weak-type estimate needs to be interpolated
  with the unweighted bound from Theorem~\ref{thm:FeffermanZygmund}. We use again    
  Theorem~\ref{thm:GrafMas} as well as the embedding of interpolation spaces 
  \begin{align}\label{eq:embedding_radial}
    \begin{aligned}
     &L^{q_0}(\R^2)^{1-\theta} (Z_{q_1})^\theta 
    \subset (L^{q_0}(\R^2),Z_{q_1})_{\theta,\infty}
    = L^{q_\theta,\infty}(\R^2) \\ 
    \text{where}\qquad&\frac{1-\theta}{q_0}+\frac{\theta}{q_1}=\frac{1}{q_\theta},\quad
    0<\theta<1,\quad  q_0\neq q_1 
   \end{aligned}
  \end{align}
  which is proved just like Proposition~\ref{prop:embedding}. We deduce the following:
  
  \begin{lem} \label{lem:weak_gamma}
     Let $\gamma>0$. Then the operator $\mathfrak R_\gamma^*: L^r(\S)\to  L^{q,\infty}(\R^2)$ is bounded  
    provided that $r\in [1,\infty],q\in (0,\infty]$ satisfy
  \begin{equation} \label{eq:weak_conditions_radial}
     \gamma\geq \max\left\{ \frac{3}{2q}-\frac{1}{2r'},\frac{2}{q}-\frac{1}{r'}\right\}
     \quad\text{and}\quad
     \gamma>  \frac{2}{q}-\frac{1}{2}.    
  \end{equation}
  \end{lem}
  \begin{proof}
  The Fefferman-Zygmund result (Theorem~\ref{thm:FeffermanZygmund}) 
  yields the boundedness $\mathfrak R_0^*:L^{r_0}(\sph{1})\to L^{q_0}(\R^2)$ and
  $\mathfrak R_{\gamma_1}^*:L^{r_1}(\sph{1})\to Z^{q_1}(\R^2)$ follows from
  Lemma~\ref{lem:weaktypebounds_radial}. Here, $r_0,r_1\in [1,\infty],q_0,q_1\in (0,\infty)$ are
  such that 
  $$
    \frac{3}{q_0}\leq \frac{1}{r_0'}, \quad \frac{1}{q_0}<\frac{1}{4},\quad 
   \gamma_1\geq \max\left\{\frac{1}{q_1},\frac{2}{q_1}-\frac{1}{r_1'},\frac{2}{q_1}-\frac{1}{2}\right\}.
  $$
  Using interpolation (Theorem~\ref{thm:GrafMas}) and $[L^{r_0}(\S),L^{r_1}(\S)]_\theta=L^{r_\theta}(\S)$ as
  well as~\eqref{eq:embedding_radial} we find that $\mathfrak R_\gamma^*:L^r(\S)\to L^{q,\infty}(\R^2)$ is
  bounded provided that  
 \begin{align*}
   \frac{1}{q}= \frac{1-\theta}{q_0}+\frac{\theta}{q_1},\quad
   \frac{1}{r}= \frac{1-\theta}{r_0}+\frac{\theta}{r_1}, \quad \gamma = \theta \gamma_1, \quad 
   0< \theta<1, \quad q_0\neq q_1 
 \end{align*}
   In the Appendix (Proposition~\ref{prop:InterpolationArithII}) we show that such a choice can be made if
   and only if \eqref{eq:weak_conditions_radial} holds and the claim is proved.
  \end{proof}

  Next we use real interpolation to upgrade these weak estimates via interpolation
  just as in the proof of Corollary~\ref{cor:mappingproperties}. In the non-endpoint case this is possible
  without additional constraint, but in the endpoint case we have to exclude the case $q=r'$ where 
  $\gamma=\frac{3}{2q}-\frac{1}{2r'}=\frac{2}{q}-\frac{1}{r'}$. One may check that precisely in this case
  it is impossible to choose $\eps>0$ as in the proof of Corollary~\ref{cor:mappingproperties}.

  \begin{cor}\label{cor:mappingproperties2}
    Let $\gamma>0$ and $0<q<\infty,1<r<\infty$. Then $\mathfrak R_\gamma^*: L^{r,s}(\S)\to L^{q,s}(\R^2)$ is
    bounded for all $s\in [1,\infty]$ provided that $\gamma>  \frac{2}{q}-\frac{1}{2}$ and
  \begin{equation} \label{eq:conditions_radial}
    \gamma> \max\left\{ \frac{3}{2q}-\frac{1}{2r'},\frac{2}{q}-\frac{1}{r'}\right\}
    \quad\text{or}\quad
    \gamma = \max\left\{ \frac{3}{2q}-\frac{1}{2r'},\frac{2}{q}-\frac{1}{r'}\right\},\;
    q\neq r'. 
  \end{equation}
  \end{cor}

  \medskip
  
  \textbf{Proof of Theorem~\ref{thm:main2} -- sufficient conditions:}  The result for $q=\infty$ is trivial
  and the one for $\gamma=0$ is already covered by Theorem~\ref{thm:FeffermanZygmund}, so assume
  $0<q<\infty$ and $\gamma>0$. Choosing $s=r$ in Corollary~\ref{cor:mappingproperties2} and exploiting the
  embeddings of Lorentz spaces yields that for any given $\gamma>0$ the operator $\mathfrak
  R_\gamma^*:L^r(\S)\to L^q(\S)$ is bounded provided that $\gamma>\frac{2}{q}-\frac{1}{2}$ and
  \begin{align*}
    \gamma&> \max\left\{ \frac{3}{2q}-\frac{1}{2r'},\frac{2}{q}-\frac{1}{r'}\right\},\;1<r<\infty
    \qquad\text{or}\\
    \gamma &= \max\left\{ \frac{3}{2q}-\frac{1}{2r'},\frac{2}{q}-\frac{1}{r'}\right\},\;
    q\neq r',\; 1<r\leq q. 
  \end{align*}
  The non-endpoint estimates actually also hold for $r\in\{1,\infty\}$ as 
  can be deduced from Lemma~\ref{lem:weak_gamma} and a short interpolation argument. 
  So the  conditions stated in Theorem~\ref{thm:main2} are sufficient. \qed

 \section{Proof of Theorem~\ref{thm:main2} -- Necessary conditions}
  
  Next we prove that the conditions in Theorem~\ref{thm:main2} cannot be improved by providing suitable
  counterexamples. As before, the constant density and Knapp-type examples are all we need. 
  For given exponents $1\leq r\leq \infty, 0<q<\infty$ and $\gamma>0$ we shall deduce  
  the necessity of~\eqref{eq:conditions_radial} as follows:
  \begin{itemize}
    \item[(i)] $\gamma>\frac{2}{q}-\frac{1}{2}$ is necessary by Lemma~\ref{lem:NecI_radial},
    \item[(ii)] $\gamma\geq \max\{\frac{3}{2q}-\frac{1}{2r'},\frac{2}{q}-\frac{1}{r'}\}$ is necessary by  
    Lemma~\ref{lem:NecII_radial},
    \item[(iii)] $\gamma=\max\{\frac{3}{2q}-\frac{1}{2r'},\frac{2}{q}-\frac{1}{r'}\}$ implies $q\neq
    r'$ by Lemma~\ref{lem:NecII_radial}~(iii).
    \item[(iv)]  $\gamma=\max\{\frac{3}{2q}-\frac{1}{2r'},\frac{2}{q}-\frac{1}{r'}\}$ implies $1<r\leq
    q$   by Lemma~\ref{lem:NecII_radial}~(i) and
    Lemma~\ref{lem:NecIII_radial}.
  \end{itemize}
  The computations are similar to the ones carried out before, so we keep the presentation as short as
  possible.

    \begin{lem}  \label{lem:NecI_radial}
      Assume $0\leq\gamma<\infty, 1\leq r\leq \infty,0<q<\infty$. If the operator $\mathfrak
      R^*_{\gamma}:L^r(\sph{1})\to L^q(\R^2)$  is bounded, then $\gamma>\frac{2}{q}-\frac{1}{2}$.
    \end{lem}
    \begin{proof}
      As in Lemma~\ref{lem:NecI} the claim follows from taking $F:= 1\in L^r(\sph{1})$ and 
    \begin{align*}
      \|\mathfrak R_{\gamma}^*F\|_{L^q(\R^2)}^q
      &= \| J_0(|(x,y)|)(1+|x|+|y|)^{-\gamma}\|_{L^q(\R^2)}^q \\
      &\stackrel{\eqref{eq:estJ0}}\gtrsim   \sum_{j=1}^\infty j^{-\frac{q}{2}} (1+z_j)^{-q\gamma}
      \big| \big\{(x,y)\in\R^2: z_j-\delta\leq |(x,y)|\leq z_j +\delta\big\}\big|      \\
      &\gtrsim \sum_{j=1}^\infty j^{-\frac{q}{2}} \cdot z_j^{1-q\gamma} \\
      &\gtrsim \sum_{j=1}^\infty j^{1-q(\frac{1}{2}+\gamma)}.
    \end{align*}
 \end{proof}

 Next we discuss the Knapp example.

 \begin{lem}  \label{lem:NecII_radial}
   Assume  $0\leq\gamma<\infty, 1\leq r\leq \infty,0<q<\infty$.  Then the following conditions
   are necessary for $\mathfrak R^*_{\gamma}:L^r(\S)\to L^q(\R^2)$ to be bounded:
    \begin{itemize}
     \item[(i)] if $\gamma=\frac{2}{q}$, then $r>1$, 
     \item[(ii)] if $\frac{2}{q}>\gamma>\frac{1}{q}$, then $\gamma\geq
     \frac{2}{q}-\frac{1}{r'}$,
     \item[(iii)] if $\gamma=\frac{1}{q}$, then $0>\frac{1}{q}-\frac{1}{r'}$,
     \item[(iv)] if $\gamma<\frac{1}{q}$, then $\gamma\geq \frac{3}{2q}-\frac{1}{2r'}$. 
   \end{itemize}
 \end{lem}
 \begin{proof}
   We mimick the proof of Lemma~\ref{lem:NecII}. With the same notation, the function $F:=
   \ind_{\mathcal C_\delta}$ satisfies $\|F\|_{L^r(\S)}  \simeq \delta^{\frac{1}{r}}$ and  for $(x,y)\in E_j$ we have
   $|(\mathcal R^*F)(x,y)|\gtrsim \delta$ as $\delta\to 0^+$ provided that $j\leq c_0 \delta^{-2}$ for some
   small constant $c_0>0$. This implies, for fixed $\gamma,q,r$ and small $\delta>0$,
   \begin{align*}
     \frac{\|\mathfrak R^*_{\gamma}F\|_{L^q(\R^2)}}{\|F\|_{L^r(\S)}}
     &\gtrsim \delta^{-\frac{1}{r}} \| (\mathcal R^*F)(x,y)
     (1+|x|+|y|)^{-\gamma}\|_{L^q(\R^2)}  \\
     &\gtrsim \delta^{-\frac{1}{r}}\,   \left(\sum_{j=1}^{\lfloor c_0\delta^{-2}\rfloor} \| \ind_{E_j}
     \delta (1+|x|+|y|)^{-\gamma}\|_{L^q(\R^2)}^q \right)^{\frac{1}{q}} \\ 
     &\simeq \delta^{-\frac{1}{r}} \left( \sum_{j=1}^{\lfloor c_0 \delta^{-2}\rfloor} \delta^q 
     \int_0^{c\delta^{-1}}  (x+j)^{-\gamma q}\,dx    \right)^{\frac{1}{q}} \\
     &\simeq \delta^{\frac{1}{r'}}   \left( 
     \sum_{j=1}^{\lfloor c_0 \delta^{-2}\rfloor}    
     \begin{cases}
       j^{1-\gamma q}- (j+\delta^{-1})^{1-\gamma q} &, \text{if }\gamma >\frac{1}{q}\\
       \log(j+\delta^{-1})-\log(j) &,\text{if } \gamma = \frac{1}{q}\\
       (j+\delta^{-1})^{1-\gamma q} - j^{1-\gamma q}  &,\text{if } \gamma<\frac{1}{q}
     \end{cases}
     \right)^{\frac{1}{q}} \\
     &\gtrsim   \delta^{\frac{1}{r'}}   \left( 
     \sum_{j=1}^{\lfloor \delta^{-1}\rfloor}    
     \begin{cases}
       j^{1-\gamma q}  &,\text{if } \gamma > \frac{1}{q}\\
       \log(\delta^{-1} j^{-1})  &,\text{if } \gamma = \frac{1}{q}\\
       \delta^{-1+\gamma q}  &,\text{if } \gamma < \frac{1}{q}\\
     \end{cases}
     \; + \sum_{j=\lceil \delta^{-1}\rceil}^{\lfloor c_0 \delta^{-2}\rfloor}
       \delta^{-1} j^{-\gamma q}   
     \right)^{\frac{1}{q}} \\
     &\simeq \delta^{\frac{1}{r'}}   \left( 
     \begin{cases}
       1 &,\text{if } \gamma >\frac{2}{q}\\
         |\log(\delta)|   &,\text{if } \gamma=\frac{2}{q}  \\
       \delta^{-2+\gamma q}  &,\text{if }\gamma<\frac{2}{q} \\
     \end{cases}
     \; +  \;
       \begin{cases}
       \delta^{-2+\gamma q}  &, \text{if }\gamma> \frac{1}{q}\\
       \delta^{-1}|\log(\delta^{-1})|  &, \text{if }\gamma = \frac{1}{q}\\
       \delta^{-3+2\gamma q}   &, \text{if }\gamma <\frac{1}{q}\\
     \end{cases}
     \right)^{\frac{1}{q}} \\
     &\simeq \delta^{\frac{1}{r'}}   \left( 
     \begin{cases}
       1 &, \text{if }\gamma >\frac{2}{q}\\
       |\log(\delta)|   &, \text{if }\gamma  = \frac{2}{q} \\
       \delta^{-2+\gamma q}  &, \text{if }\frac{2}{q}>\gamma>\frac{1}{q} \\
       \delta^{-1}|\log(\delta)| &, \text{if }\gamma=\frac{1}{q}\\
       \delta^{-3+2\gamma q} &,\text{if } \gamma<\frac{1}{q}
     \end{cases}
     \right)^{\frac{1}{q}}. 
   \end{align*}
   This implies the claim as $\delta\to 0^+$.
   Note that in the second last estimate for $\gamma=\frac{1}{q}$ we used 
   $$
     \sum_{j=1}^{\lfloor \delta^{-1}\rfloor}   \log(\delta^{-1} j^{-1}) 
     \simeq \int_1^{\delta^{-1}} \log(\delta^{-1}x^{-1})\,dx 
     \simeq \delta^{-1}.
   $$ 
    
 \end{proof} 
 
 \begin{lem}   \label{lem:NecIII_radial}
   Assume $0<\gamma<\infty$ where
   $\gamma=\max\{\frac{3}{2q}-\frac{1}{2r'},\frac{2}{q}-\frac{1}{r'}\}>\frac{2}{q}-\frac{1}{2}$ and 
    $0<q< \infty,1<r\leq \infty$. Then $\mathfrak R^*_{\gamma}:L^r(\S)\to L^q(\R^2)$ can only be bounded if
   $q\geq r$.
 \end{lem}
 \begin{proof}    
   The claim is trivial for $\gamma=\frac{3}{2q}-\frac{1}{2r'}$
   because $\gamma>\frac{2}{q}-\frac{1}{2}$ already implies $q>r$. So assume 
   $\gamma=\frac{2}{q}-\frac{1}{r'}$. It suffices to prove the
   claim assuming a priori $q\geq 1$. We then prove the equivalent statement that the
   dual bound 
   $$
      \|\hat f\|_{L^{r'}(\S)}  \lesssim \|(1+|x|+|y|)^{\gamma} f\|_{L^{q'}(\R^2)}
  $$
  can only hold for $q\geq r$. Set $f_\eps(x,y):=(1+|x|^2+|y|^2)^{-\delta/2} e^{iy}$ for
  $\delta:=2-\frac{1}{r'}+\eps = \gamma+\frac{2}{q'}+\eps$. Note that $1<\delta<2$.
  Then 
  \begin{align*}
     \|(1+|x|+|y|)^{\gamma}  f_\eps\|_{L^{q'}(\R^2)}
     &\les  \|(1+|x|+|y|)^{\gamma-\delta} \|_{L^{q'}(\R^2)} \\
     &\les  \left( \int_0^\infty r (1+r)^{q'(\gamma-\delta)}  \,dr  \right)^{\frac{1}{q'}}  \\
     &\les  \left(\frac{1}{-2-q'(\gamma-\delta)}\right)^{\frac{1}{q'}} \\
     &\les \eps^{-\frac{1}{q'}}
  \end{align*}
  and the same bound holds in the case $q'=\infty$. 
  Proposition~\ref{prop:oscint}  yields as $(\xi_1,\xi_2)\to (0,1)$
   \begin{align*}
     |\hat f_\eps(\xi)|
     &= \left| \int_{\R^2} (1+|x|^2+|y|^2)^{-\delta/2} e^{-i(\xi_1,\xi_2-1)\cdot (x,y)}\,dx \,dy \right| \\
     &= \left| \int_{\S} \int_0^\infty r(1+r^2)^{-\delta/2} e^{-ir\omega\cdot (\xi_1,\xi_2-1)}\,dr
     \,d\sigma(\omega) \right|   \\
     &\gtrsim  \int_{\S} |\omega\cdot (\xi_1,\xi_2-1)|^{\delta-2}  \,d\sigma(\omega)    \\
     &\simeq  |(\xi_1,\xi_2-1)|^{\delta-2}      \\
     &\simeq  (|\xi_1|+|\xi_2-1|)^{-\frac{1}{r'}+\eps}
   \end{align*}
   So the validity of the dual estimate stated above and
   $r'<\infty$ implies for some small $c>0$
   \begin{align*}
     \eps^{-\frac{1}{q'}}
     &\gtrsim \|\hat f_\eps\|_{L^{r'}(\S)}
     \;=\; \left( \int_0^{2\pi} |\hat f_\eps(\sin(t),\cos(t))|^{r'}\,dt \right)^{\frac{1}{r'}}\\
     &\gtrsim \left( \int_0^{c}
     (|\sin(t)|+|\cos(t)-1|)^{r'(-\frac{1}{r'}+ \eps)}  \,dt \right)^{\frac{1}{r'}}
     \;\gtrsim\; \left( \int_0^{c} t^{-1+ r'\eps} \,dt \right)^{\frac{1}{r'}} \\
     &\simeq  \eps^{-\frac{1}{r'}}.
   \end{align*}
   From this we infer $r'\geq q'$, hence $q\geq r$.
 \end{proof}

\section*{Appendix 1 -- An oscillatory intgral}
   
   We first analyze the asymptotics of oscillatory integrals
   $$
     \int_0^\infty \chi(r)r^{-\kappa}\cos(\lambda r)\,dr
   $$
   as $\lambda\to 0^+$ under reasonable assumptions on $\chi$ and $\kappa$.
 
 \begin{prop}\label{prop:oscint}
   Let $\chi:[0,\infty)\to\R$ be measurable with $|\chi(z)-\chi_\infty|\leq C(1+z)^{-\sigma}$ where
   $\chi_\infty>0$ and $0<\kappa<1<\kappa+\sigma$.  Then 
   $$
     \lim_{\lambda\to 0^+} \lambda^{1-\kappa} \int_0^\infty \chi(r)r^{-\kappa} \cos(\lambda r)\,dr
     = \chi_\infty \int_0^\infty \rho^{-\kappa}\cos(\rho)\,d\rho 
     \in (0,\infty)
   $$
   locally uniformly with respect to $\kappa$. In particular,
   $$
       \int_0^\infty \chi(r)r^{-\kappa}\cos(\lambda r)\,dr
     \gtrsim \lambda^{\kappa-1} \qquad\text{as }\lambda \to 0^+.
   $$
 \end{prop}
 \begin{proof}
   We have for all $\lambda>0$
   \begin{align*}
      &\;\left|\lambda^{1-\kappa} \int_0^\infty \chi(r)r^{-\kappa}\cos(\lambda r)\,dr 
      - \chi_\infty \int_0^\infty \rho^{-\kappa}\cos(\rho)\,d\rho \right|  \\
     &=  \left|   \int_0^\infty \rho^{-\kappa}\big( \chi(
     \rho\lambda^{-1}) - \chi_\infty\big) \cos(\rho)\,d\rho \right|   \\
     &\leq  C \int_0^\infty \rho^{-\kappa}(1+ \rho \lambda^{-1})^{-\sigma}\,d\rho \\
     &=  C\lambda^{1-\kappa}  \int_0^\infty s^{-\kappa}(1+ s)^{-\sigma}\,ds.  
   \end{align*}
   This term converges zu zero as $\lambda\to 0^+$. So it remains to show    
   $$
     \int_0^\infty \rho^{-\kappa} \cos(\rho)\,d\rho  >0.
   $$
  Roughly speaking, it is sufficient to prove that the positivity regions of the integrand
 dominate the negativity regions. Formally, we prove for all $k\in\N_0$ 
 $$ 
   \int_{[2k\pi,(2k+\frac{1}{2})\pi]\cup [(2k+\frac{3}{2})\pi,(2k+2)\pi]}
    s^{-\kappa} \cos(s)\,ds
   > \left|\int_{[(2k+\frac{1}{2})\pi,(2k+\frac{3}{2})\pi]} s^{-\kappa} \cos(s)\,ds \right|.
 $$
 Write $\mu:=2\pi k$. Given the properties of cosine this is equivalent to
 $$ 
   \int_0^{\frac{\pi}{2}}
    \Big((\mu+s)^{-\kappa}+(\mu+2\pi-s)^{-\kappa}\Big) \cos(s)\,ds
    > \int_0^{\frac{\pi}{2}}  \Big((\mu+\pi+s)^{-\kappa}+(\mu+\pi-s)^{-\kappa}\Big)
    \cos(s)\,ds
 $$
 and the latter statement follows from    
 $$
   \cos(s)\Big(  (\mu+s)^{-\kappa}+(\mu+2\pi-s)^{-\kappa}-(\mu+\pi+s)^{-\kappa}-(\mu+\pi-s)^{-\kappa} \Big)
   >0 \quad\text{for } 0<s<\frac{\pi}{2}. 
 $$
 Indeed,  the first factor is positive on $(0,\pi/2)$ and second factor decreases to zero on the interval
 $(0,\frac{\pi}{2}]$.  
 \end{proof}

\section{Appendix 2 --  Interpolation arithmetic}
 
 The following result was used in the proof of Lemma~\ref{lem:weaktypebounds_smallalpha}.
 
 \begin{prop}\label{prop:InterpolationArithI}
   Assume $1\leq r<\infty$ and $0<q< \infty$. Then there are $r_0\in [1,\infty],r_1\in [1,\infty),
  q_0\in [1,\infty], q_1\in (0,\infty]$ and $\theta\in (0,1)$ such that
   \begin{align*} 
     \begin{aligned}
     &\frac{1}{q}=\frac{1-\theta}{q_0}+\frac{\theta}{q_1},\quad
     \frac{1}{r}=\frac{1-\theta}{r_0}+\frac{\theta}{r_1}\qquad \text{and} \\
     &\frac{1}{q_0} \leq \frac{1}{3r_0'},\quad\frac{1}{q_0}<\frac{1}{4},\quad
     \frac{\alpha}{\theta} \b{=} \frac{1}{q_1},\quad
     \frac{\beta}{\theta}\geq \frac{1}{q_1}-\frac{1}{2r_1'},\quad q_0\neq q_1
     \end{aligned}
   \end{align*}
   if and only if
  $$
     \alpha+\beta>\frac{2}{q}-\frac{1}{2}\quad\text{and}\quad 
     \alpha+2\beta\geq \frac{3}{q}-\frac{1}{r'}.
   $$ 
\end{prop}
  \begin{proof}
   Setting $q_1:=\frac{\theta}{\alpha}$ the given conditions are  equivalent to
   $$
     \frac{1}{q}-\alpha=\frac{1-\theta}{q_0},\;
     \frac{1}{r'}=\frac{1-\theta}{r_0'}+\frac{\theta}{r_1'},\quad
     \frac{1}{q_0}\leq\frac{1}{3r_0'},\;\frac{1}{q_0}<\frac{1}{4},\; 
     \beta \geq \alpha -\frac{\theta}{2r_1'},\; q_0\neq \frac{\theta}{\alpha}.
   $$
   We may choose $r_0\in [1,\infty]$ according to the second equation  if and only if
   $$
     \frac{1}{q}-\alpha=\frac{1-\theta}{q_0},\;
     1-\theta\geq \frac{1}{r'}-\frac{\theta}{r_1'}\geq \frac{3(1-\theta)}{q_0},\;
     \frac{1}{q_0}<\frac{1}{4},\; 
     \beta \geq \alpha -\frac{\theta}{2r_1'},\; q_0\neq \frac{\theta}{\alpha}.
   $$
   Choosing $q_0\in [1,\infty]$ according to these conditions is possible if and only if 
   $$
     \frac{1}{q}-\alpha<\frac{1-\theta}{4},\quad
     1-\theta\geq \frac{1}{r'}-\frac{\theta}{r_1'}\geq \frac{3}{q}-3\alpha,\quad      
     \beta \geq \alpha -\frac{\theta}{2r_1'},\quad\theta\neq  \alpha q.
   $$
   This is equivalent to
   $$
	\theta< 1-\frac{4}{q}+4\alpha,\quad
     \frac{1}{r'}-\frac{3}{q}+3\alpha\geq \frac{\theta}{r_1'}  
     \geq \max\left\{2\alpha -  2\beta,\theta-\frac{1}{r}\right\},\quad
     \theta\neq  \alpha q.
   $$
   A choice of $r_1\in [1,\infty)$ is possible if and only if
   $$
     1-\frac{4}{q}+4\alpha>
     \theta> \max\left\{2\alpha -  2\beta,\theta-\frac{1}{r}\right\},\quad 
     \frac{1}{r'}-\frac{3}{q}+3\alpha\geq  \max\left\{2\alpha -  2\beta,\theta-\frac{1}{r}\right\},\quad
     \theta\neq \alpha q.
   $$
   This is equivalent to
   $$
     2\alpha-2\beta<\theta<1-\frac{4}{q}+4\alpha,\quad
     \theta\leq 1-\frac{3}{q}+3\alpha,\quad
     \alpha+2\beta\geq \frac{3}{q}-\frac{1}{r'}, \quad
     \theta\neq \alpha q  
   $$    
   and hence, in view of $\alpha<\frac{1}{q}$, equivalent to
   $$
     \alpha+\beta>\frac{2}{q}-\frac{1}{2},\quad
     \alpha+2\beta\geq \frac{3}{q}-\frac{1}{r'}. 
   $$      
 \end{proof}
 
 \begin{prop} \label{prop:InterpolationArithII}
   Assume $1\leq r\leq \infty$ and $0<q< \infty$. Then there  are $r_0,r_1\in [1,\infty],q_0\in
   [1,\infty],q_1\in (0,\infty)$ such that
 \begin{align*}
   \frac{1}{q}= \frac{1-\theta}{q_0}+\frac{\theta}{q_1},\quad
   \frac{1}{r}= \frac{1-\theta}{r_0}+\frac{\theta}{r_1}, \quad 
   \frac{3}{q_0}\leq \frac{1}{r_0'}, \quad \frac{1}{q_0}<\frac{1}{4},\quad q_0\neq q_1 \\
   \gamma = \theta \gamma_1,\quad 
   \gamma_1\geq \max\left\{\frac{1}{q_1},\frac{2}{q_1}-\frac{1}{r_1'},\frac{2}{q_1}-\frac{1}{2}\right\}, \quad 
   0< \theta<1
 \end{align*}
 if and only if 
  $$
     \gamma\geq \max\left\{ \frac{3}{2q}-\frac{1}{2r'},\frac{2}{q}-\frac{1}{r'}\right\},\quad
     \gamma>  \frac{2}{q}-\frac{1}{2}.    
  $$
  \end{prop}
  \begin{proof}
    The conditions may be recast as
 \begin{align*}
   \frac{1}{q}= \frac{1-\theta}{q_0}+\frac{\theta}{q_1},\quad
   \frac{1}{r'}= \frac{1-\theta}{r_0'}+\frac{\theta}{r_1'},  \quad
   \frac{3}{q_0}\leq \frac{1}{r_0'}, \quad \frac{1}{q_0}<\frac{1}{4}, \\  
   \gamma\geq
   \max\left\{\frac{\theta}{q_1},\frac{2\theta}{q_1}-\frac{\theta}{r_1'},\frac{2\theta}{q_1}-\frac{\theta}{2}\right\},
   \quad 0<\theta<1.
 \end{align*}
 Substituting $r_1$ we obtain the equivalent conditions
 \begin{align*}
   \frac{1}{q}= \frac{1-\theta}{q_0}+\frac{\theta}{q_1},\quad
   0\leq \frac{1}{r'}-\frac{1-\theta}{r_0'}\leq  \theta,  \quad
   \frac{3}{q_0}\leq \frac{1}{r_0'}, \quad\frac{1}{q_0}<\frac{1}{4}, \\  
   \gamma\geq
   \max\left\{\frac{\theta}{q_1},\frac{2\theta}{q_1}-\frac{1}{r'}+\frac{1-\theta}{r_0'},\frac{2\theta}{q_1}-\frac{\theta}{2}\right\},
    \quad 0< \theta<1
  \end{align*}
  and hence
 \begin{align*}
   \frac{1}{q}= \frac{1-\theta}{q_0}+\frac{\theta}{q_1},\quad
    \max\left\{\frac{3(1-\theta)}{q_0}, \frac{1}{r'}-\theta\right\}\leq \frac{1-\theta}{r_0'}\leq 
    \frac{1}{r'}, \quad \frac{1}{q_0}<\frac{1}{4}, \\  
   \gamma\geq
   \max\left\{\frac{\theta}{q_1},\frac{2\theta}{q_1}-\frac{1}{r'}+\frac{1-\theta}{r_0'},\frac{2\theta}{q_1}-\frac{\theta}{2}\right\},
    \quad 0< \theta<1.
 \end{align*}
 Choosing $r_0$ as small as possible leads to the equivalent set of conditions
 \begin{align*}
   \frac{1}{q}= \frac{1-\theta}{q_0}+\frac{\theta}{q_1},\quad
    \frac{1-\theta}{q_0} \leq  \frac{1}{3r'},  \quad
    \frac{1-\theta}{q_0}<\frac{1-\theta}{4}, \quad 0< \theta<1, \\  
   \gamma\geq
   \max\left\{\frac{\theta}{q_1},
   \frac{2\theta}{q_1}-\theta,
   \frac{2\theta}{q_1}-\frac{1}{r'}+\frac{3(1-\theta)}{q_0},
   \frac{2\theta}{q_1}-\frac{\theta}{2}\right\}.
 \end{align*}
 We now choose $q_1\in (0,\infty]$ according to the first equation and obtain
 \begin{align*}
    \frac{1-\theta}{q_0}\leq \min\left\{\frac{1}{q},\frac{1}{3r'}\right\},  \quad
    \frac{1-\theta}{q_0}<\frac{1-\theta}{4}, \quad 0< \theta<1, \\  
   \gamma\geq
   \max\left\{ \frac{1}{q}-\frac{1-\theta}{q_0},
   \frac{2}{q} -\frac{1}{r'}+\frac{1-\theta}{q_0},
   \frac{2}{q}-\frac{2(1-\theta)}{q_0}-\frac{\theta}{2}\right\},
 \end{align*}
 which is equivalent to
 \begin{align*}
    \max\left\{\frac{1}{q}-\gamma, \frac{1}{q}-\frac{\gamma}{2}-\frac{\theta}{4}\right\}
    \leq \frac{1-\theta}{q_0}
    \leq \min\left\{\frac{1}{q},\frac{1}{3r'},\gamma-\frac{2}{q}+\frac{1}{r'}\right\},  \;\;
    \frac{1-\theta}{q_0}<\frac{1-\theta}{4}, \;\; 0< \theta<1.  
 \end{align*}
 A choice of $q_0$ according to these conditions is possible if and only if  
 \begin{align*}
    \max\left\{\frac{1}{q}-\gamma, \frac{1}{q}-\frac{\gamma}{2}-\frac{\theta}{4},0\right\}
    \leq \min\left\{\frac{1}{q},\frac{1}{3r'},\gamma-\frac{2}{q}+\frac{1}{r'}\right\},  \\
    \max\left\{\frac{1}{q}-\gamma, \frac{1}{q}-\frac{\gamma}{2}-\frac{\theta}{4},0\right\}<\frac{1-\theta}{4}, \quad 0<
    \theta<1,
 \end{align*} 
 which is equivalent to  
 \begin{align*}    
     \max\left\{\frac{4}{q}-2\gamma-\frac{4}{3r'},\frac{12}{q}-6 \gamma -\frac{4}{r'}\right\}
     \leq \theta < 1-\frac{4}{q}+4\gamma ,\quad 0<\theta<1\\
     \gamma\geq \max\left\{\frac{2}{q}-\frac{1}{r'}, \frac{3}{2q}-\frac{1}{2r'}\right\},  \quad 
     \gamma>\frac{2}{q}-\frac{1}{2}.   
 \end{align*}
 This simplifies to
 \begin{align*}
    \gamma\geq \max\left\{\frac{2}{q}-\frac{1}{r'}, \frac{3}{2q}-\frac{1}{2r'}\right\},  \quad
    \gamma>\frac{2}{q}-\frac{1}{2},   
 \end{align*}
 which is all we had to show. 
  \end{proof}

\section*{Acknowledgments}

I wish to thank L. Grafakos and M. Masty{\l}o for helpful clarifications regarding~\cite{GrafakosMastylo}.

\bibliographystyle{abbrv}
\bibliography{biblio}

\end{document}